\documentclass{article}
\usepackage{amsrefs}
\usepackage{amssymb}
\usepackage{amsmath}
\usepackage{amsthm}
\usepackage{graphicx}

\DeclareMathOperator{\Pic}{Pic}

\DeclareMathOperator{\rk}{rank}
\DeclareMathOperator{\codim}{codim}
\DeclareMathOperator{\im}{im}

\DeclareMathOperator{\Fitt}{Fitt}

\DeclareMathOperator{\Sm}{Sm}
\DeclareMathOperator{\Sing}{Sing}
\DeclareMathOperator{\Supp}{Supp}
\newtheorem{theorem}{Theorem}
\newtheorem*{theoremA}{Theorem A}
\newtheorem*{theoremB}{Theorem B}
\newtheorem*{theoremC}{Theorem C}
\newtheorem{proposition}{Proposition}
\newtheorem{lemma}{Lemma}
\newtheorem{definition}{Definition}
\newtheorem{corollary}{Corollary}
\newtheorem*{subject}{2000 Mathematics Subject Classification}
\newtheorem*{keywords}{Keywords}
\theoremstyle{definition}
\newtheorem{example}{Example}

\theoremstyle{remark}
\newtheorem{notation}{Notation}
\newtheorem{remark}{Remark}

\author{Marc Coppens\footnote{KU  Leuven, Department of Mathematics, Section of Algebra,
Celestijnenlaan 200B bus 2400 B-3001 Leuven and Technologiecampus Geel, department Elektrotechniek (ESAT),Kleinhoefstraat 4, B-2440 Geel Belgium; email: marc.coppens@kuleuven.be.}}
\title{Smoothness results for the schemes of special divisors on general $k$-gonal curves}
\date{}

\begin{document}
\maketitle \noindent

\begin{abstract}
For a general $k$-gonal curve $C$ with a morphism $f: C \rightarrow \mathbb{P}^1$ of degree $k$, we consider the refinement of the Brill-Noether schemes $W^r_d(C)$ by means of the Brill-Noether degeneracy schemes $\overline{\Sigma}_{\overrightarrow {e}}(C,f)$.
The schemes  $\overline{\Sigma}_{\overrightarrow {e}}(C,f)$ as sets are closures of subsets  $\Sigma_{\overrightarrow {e}}(C,f)$ of $\Pic (C)$ and as a scheme  $\Sigma_{\overrightarrow {e}}(C,f)$ is a smooth open subscheme of  $\overline{\Sigma}_{\overrightarrow {e}}(C,f)$.
In this paper we describe naturally defined open subsets of  $\overline{\Sigma}_{\overrightarrow {e}}(C,f)$ in general strictly containing  $\Sigma_{\overrightarrow {e}}(C,f)$  such that  $\overline{\Sigma}_{\overrightarrow {e}}(C,f)$ is smooth along them.

As an application we describe all invertible sheaves $L$ on $C$ having an injective Petri map.
Some of those sets  $\overline{\Sigma}_{\overrightarrow {e}}(C,f)$ are the irreducible components of $W^r_d(C)$.
In those cases we prove $W^r_d(C)$ is smooth at a point $L$ of those larger open subsets of  $\overline{\Sigma}_{\overrightarrow {e}}(C,f)$ unless $L$ belongs to at least two irreducible components of $W^r_d(C)$ (such points exist).
On the other hand in general the singular locus of the schemes $W^r_d(C)$ is not equal to the complement of the union of $W^{r+1}_d(C)$ and the intersections of two different components of $W^r_d(C)$.
\end{abstract}

\begin{subject}
14H51
\end{subject}

\begin{keywords}
Hurwitz-Brill-Noether theory, splitting degeneracy loci, special divisors, Brill-Noether schemes, gonality, smoothness, Petri map
\end{keywords}

\section{Introduction}\label{section1}

Let $C$ be a smooth irreducible projective curve of genus $g$ defined over an algebraically closed field $F$ (of arbitrary characteristic).
Let $L$ be a line bundle on $C$ and let $K$ be the canonical bundle on $C$.
The natural multiplication map
\[
\mu _L : H^0(C,L) \otimes H^0(C,K-L) \rightarrow H^0(C,K)
\]
is called the Petri map of $L$.
In case $h^0(L)=r+1$ with $r \in \mathbb{Z}_{\geq 0}$, $\deg (L)=d$ and $h^0(K-L) \geq 1$ then $L$ belongs to the Brill-Noether subscheme $W^r_d(C)$ of $\Pic ^d(C)$ and there is a natural identification between the tangent space $T_L(W^r_d(C))$ of $W^r_d(C)$ at $L$ and $(\im (\mu _L))^{\perp}$ being the orthogonal complement of $\im \mu_L$ with respect to the Serre duality paring
\[
< , >: H^0(C,K) \times H^1(C,\mathcal{O}_C) \rightarrow F
\]
(see \cite{ref9}, Chapter IV, Proposition (4.2)(i); although in this book one assumes $F=\mathbb{C}$ those arguments work for any choice of $F$).
In this statement the scheme structure of $W^r_d(C)$ is defined as in Definition \ref{schemespecialdivisors}.
In case $C$ is a general curve (meaning belonging to a suited non-empty open subset of the irreducible moduli space $M_g$ of curves of genus $g$) then the so-called Gieseker-Petri theorem states that $\mu_L$ is injective for all line bundles $L$ on $C$.
The first proof of this theorem occurs in \cite{GIES}.
This theorem of Petri-Gieseker is equivalent to the statements
\[
\dim (W^r_d(C))=\rho ^r_d(g)=g-(r+1)(g-d+r)
\]
(and $W^r_d(C)=\emptyset$ if $\rho ^r_d(g)<0$) and
\[
\Sing (W^r_d(C))=W^{r+1}_d(C) \text { .}
\]
The first statement is the so-called Brill-Noether theorem (with its first proof in \cite {GHBNT}), the second statement is with respect to the scheme structure on $W^r_d(C)$ mentioned before.

The gonality of a smooth curve $C$ is the smallest integer $k$ such that there exists a morphism $f:C \rightarrow \mathbb{P}^1$ of degree $k$.
Equivalently, it is the smallest integer $k$ such that $W^1_k(C) \neq \emptyset$.
For a general curve $C$ of genus $g$ the Brill-Noether theorem implies the gonality is equal to $[\frac{g+3}{2}]$ (meaning the largest integer being at most $\frac{g+3}{2}$).
In this paper we consider $k$-gonal curves $C$ of genus $g$ with $k<[\frac{g+3}{2}]$.
From \cite{ref15} it follows the sublocus of $M_g$ of $k$-gonal curves is irreducible; by a general $k$-gonal curve we mean a curve contained in a suited non-empty open subset of that locus.
Let $M$ be the line bundle of degree $k$ on C with $h^0(M)=2$.
Since $k<[\frac{g+3}{2}]$ it follows the Petri map of $M$ is not injective.
It is proved in \cite{AC} that $M$ is unique for a general $k$-gonal curve.
With respect to the injectivity of the Petri map, we obtain the following result for a general $k$-gonal curve.
\begin{theoremA}
Let $C$ be a general $k$-gonal curve of genus $g$ with $k<[\frac{g+3}{2}]$ and let $L$ be a line bundle on $C$ with $h^0(L)\cdot h^0(K-L) \neq 0$.
Then the Petri map $\mu_L$ of $L$ is injective if and only if $h^0(L-M)\cdot h^0(K-L-M)=0$.
\end{theoremA}
It is very easy to see that the condition is necessary, the statement is that the condition is also sufficient in case $C$ is a general $k$-gonal curve.
Line bundles satisfying $h^0(L)\cdot h^0(K-L) \neq 0$ are called special line bundles, because of the Riemann-Roch theorem they are the only interesting one in the context of this paper.
Let $L \in W^r_d(C) \setminus W^{r+1}_d(C)$ be a special line bundle having an injective Petri map, then locally at $L$ the dimension of $W^r_d(C)$ is equal to $\rho ^r_d(g)$ (it also implies $\rho ^r_d(g) \geq 0$).
In particular $W^r_d(C)$ has at least one component $Z$ of dimension $\rho^r_d(g)$ if some $L\in W^r_d(C) \setminus W^{r+1}_d(C)$ has an injective Petri map.
It is easy to see that the condition $h^0(L-M)=0$ implies $r \leq k-1$ (see e.g. \cite{ref2}, Lemma 3).
In \cite{CM} it is proved that in case $r\leq k-1$ and $\rho ^r_d(g) \geq 0$ then a general $k$-gonal curve $C$ of genus $g$ has an irreducible component $Z$ of $W^r_d(C)$ of dimension $\rho^r_d(g)$.
In \cite{ref6} it is proved that this irreducible component is unique under the condition $h^0(L-M)=0$ if $\rho ^r_d(g)>0$.
In case $\rho^r_d(g)=0$ then it is proved in \cite {ref7} that there are finitely many line bundles $L \in W^r_d(C)$ with $h^0(L-M)=0$ and they are isolated points of $W^r_d(C)$.
In \cite {ref7} those irreducible components of $W^r_d(C)$ are called of type I (see also Remark \ref{types}).

In case $C$ is a general $k$-gonal curve, using Theorem A we can determine which points $L \in Z$ ($Z$ being a component of type I) do belong to $\Sing (W^r_d(C))$.
Of course $Z \cap W^{r+1}_d(C) \subset \Sing (W^r_d(C))$ but in general this inclusion is strict.
In many cases, if $L \in Z \cap \Sing (W^r_d(C))$ and $L \notin W^{r+1}_d(C)$ then there exists a component $Z' \neq Z$ of $W^r_d(C)$ with $L \in Z'$, but we also find situations where such $Z'$ does not exist.
In particular the reduced space $Z$ is not always smooth outside $W^{r+1}_d(C)$.
For the other irreducible components of $W^r_d(C)$ we prove the following theorem (with the ingredient $V^r_{d,l}(C,f)$ explained below).

\begin{theoremB}
Let $C$ be a general $k$-gonal curve and let $L \in V^r_{d,l}(C,f)$.
Then the scheme $W^r_d(C)$ is smooth at $L$ if and only if $L$ is contained in a unique component of $W^r_d(C)$.
\end{theoremB}

As a set $V^r_{d,l}(C,f)$ is a non-empty open (hence dense) subset of an irreducible component of $W^r_d(C)$, however their scheme structures have different definitions.
In the proof of Theorem B we identify the points $L \in V^r_{d,l}(C,f)$ such that the schemes $V^r_{d,l}(C,f)$ and $W^r_d(C)$ have the same tangent space at $L$.
In case this does not hold we show $L$ is contained in at least two components.
See Example \ref{lastEx} illustrating Theorem B.

In \cite{ref7} the author considers a finer stratification of $\Pic^d(C)$ using Brill-Noether splitting loci $\Sigma _{\overrightarrow {e}}(C,f)$ associated to a morphism $f : C \rightarrow \mathbb {P}^1$ of degree $k$ (see Definition 
\ref{def1.2}).
The author also introduces so-called Brill-Noether degeneracy loci $\overline{\Sigma}_{\overrightarrow {e}}(C,f)$ (see Definition \ref{def1.4}) and in case $C$ is a general $k$-gonal curve it is the closure of $\Sigma_{\overrightarrow {e}}(C,f)$.
Irreducible components of the schemes $W^r_d(C)$ for a general $k$-gonal curve $C$ are certain $\overline{\Sigma}_{\overrightarrow {e}}(C,f)$ with $\overrightarrow {e}$ of balanced plus balanced type (see Definition \ref{balplusbal}).
In \cite {ref6} a scheme structure is defined on $\overline {\Sigma}_{\overrightarrow {e}}(C,f)$ (see Definition \ref{schemedegloci}) and this scheme is smooth along $\Sigma _{\overrightarrow {e}}(C,f)$ (see \cite{ref2}, Proposition 3).

Now assume $\overline {\Sigma}_{\overrightarrow {e}}(C,f)$ is an irreducible component of some scheme $W^r_d(C)$ for a general $k$-gonal curve $C$.
It is proved in \cite {ref2} that the scheme $W^r_d(C)$ is smooth along $\Sigma_{\overrightarrow {e}}(C,f)$.
Recently, in that case, in \cite{FFR} one introduces a naturally defined larger open subset $V^r_{d,l}(C,f)$ of $\overline {\Sigma}_{\overrightarrow {e}}(C,f)$ containing $\Sigma _{\overrightarrow {e}}(C,f)$ (see also Definition \ref{defV}).
See also Remark \ref {remV} including the relation between the parameter $l$ and $\overrightarrow {e}$.
If it would be known that $\overline{\Sigma}_{\overrightarrow {e}}(C,f)$ is smooth along $V^r_{d,l}(C,f)$ then Theorem A would be an immediate corollary of it.
However, as explained in Remark \ref {remV}, the smoothness results on $V^r_{d,l}(C,f)$ obtained in \cite {FFR} for special types of $k$-gonal curves do not imply smoothness results on $V^r_{d,l}(C,f)$ for general $k$-gonal curves.
As a partial generalization of Theorem A we prove the following result.

\begin{theoremC}
Let $C$ be a general $k$-gonal curve.
Let $\overline{\Sigma}_{\overrightarrow {e}}(C,f)$ be an irreducible component of some scheme $W^r_d(C)$.
The scheme $\overline{\Sigma}_{\overrightarrow{e}}(C,f)$ is smooth along $V^r_{d,l}(C,f)$ (with $l$ computed from $\overrightarrow {e}$ as explained in Remark \ref{remV}).
\end{theoremC}

The proofs of the theorems make use of the smoothness results from \cite{ref7}.
The smoothness of the scheme $\overline{\Sigma}_{\overrightarrow {e}}(C,f)$ along $\Sigma _{\overrightarrow {e}}(C,f)$ in case $C$ is a general $k$-gonal curve follows from the injectivity of some Petri map for splitting loci.
However the scheme strucure of those Brill-Noether degeneracy loci, using multiple Fitting ideals, do not have the same definition as the scheme structure on the Brill-Noether schemes $W^r_d(C)$ (defined by exactly one Fitting ideal).

In order to use the smoothness result from \cite{ref7} we limit the number of Fitting ideals needed to define the scheme structure of the Brill-Noether splitting degeneracy loci locally at certain points of Brill-Noether splitting loci.
To do so we make a stepwise description of those Brill-Noether splitting loci as degeneracy loci of certain bundle maps on suited open subsets of the Brill-Noether splitting degeneracy loci.
This is a generalisation of Proposition 2 in \cite{ref2}.
It should be noted that in \cite{FFR} a proof of that proposition in \cite{ref2} is obtained using a much easier bundle map, however we use bundle maps defining the Fitting ideals needed for the scheme structure while this relation is not clear (to me) with the bundle map used in \cite{FFR}.
In particular it is not clear to me how to obtain the smoothness results in Section 4 of \cite{ref2} from the bundle map used in \cite{FFR}.

While the injectivity of the Petri map mentioned in Theorem A is used mostly to conclude smoothness, we prove smoothness to conclude the injectivity of the Petri map.
However the injectivity of the Petri map of $L$ implies more than the smoothness of $W^r_d(C)$ at some $L \in W^r_d(C) \setminus W^{r+1}_d(C)$.
As a matter of fact, in case $L\in W^{r+1}_d(C)$ we have $L \in \Sing (W^r_d(C))$ and it is interesting to know more about this singularity.
It is proved in \cite{ref16} that a lot of information on it follows from the injectivity of the Petri map of $L$ in case $F$ has characteristic 0.

In Section 2 we recall some ingredients we are going to use in the proofs of the theorems.
In particular we recall the scheme structures on $W^r_d(C)$ and $\overline {\Sigma}_{\overrightarrow {e}}(C,f)$.
We also recall some of the main theorems of Hurwitz-Brill-Noether Theory.

In Section 3 we obtain a proof of Theorem C (see Theorem \ref{TheoC}).
It contains the basic result of this paper (Proposition \ref{prop1}) which reduces the number of Fitting ideals needed to define the scheme structure of $\overline{\Sigma}_{\overrightarrow {e}}(C,f)$ at certain points.
In particular Corollary  \ref{cor3.1} is a much more general statement than Theorem C.

In Section 4 we prove Theorem A (see Theorem \ref{TheoA}).
We also find that for a general $k$-gonal curve it is possible to have $L \in W^r_d(C) \setminus W^{r+1}_d(C)$ and $W^r_d(C)$ singular at $L$.
This can be caused because $L$ is contained in more than one component of $W^r_d(C)$ (see Proposition \ref{prop4}; this can be called a trivial reason) but it also occurs when $L$ is contained is a unique component of $W^r_d(C)$ (see Proposition \ref{prop5} and some examples).

In Section 5 we show how Petri maps can be used to generalize a smoothness result contained in \cite {ref6} (see Proposition \ref{prop5.1}).
Similar arguments are used to prove Theorem B (see Theorem \ref{theorem5.1} which is a more concrete statement).

\subsection{Some conventions and notations}

As this paper can be considered as a continuation of \cite{ref2} we continue to use the notations from that paper.
In particular, if $\mathcal{F}$ is a sheaf on a scheme $T$ of finite type over $F$ and $x$ is a closed point of $T$, then $\mathcal{F}_x$ is the stalk and $\mathcal{F}(x)=\mathcal{F}_x \otimes _{\mathcal{O}_{T,x}} F(x)$ is called the fibre at $x$.

Let $X$ be a scheme of finite type defined over $F$.
Unless otherwise stated, by a point on $X$ we mean a closed point.
We write $\Sing (X)$ to denote the set of the singular points of $X$ and $\Sm (X)$ to denote the set of the smooth points of $X$.
In case $\mathcal{I}$ is a sheaf of $\mathcal{O}_X$-ideals then we write $Z(\mathcal{I})$ to denote the closed subscheme of $X$ defined by $\mathcal{I}$ (hence it is $\Supp (\mathcal{O}_X / \mathcal {I})$ with structure sheaf $\mathcal {O}_X / \mathcal{I}$).
We say $X$ is irreducible at $p \in X$ in case $p$ is not contained in two different irreducible components of $X$.
In that case, the dimension $\dim_p(X)$ of $X$ at $p$ is the dimension of the irreducible component of $X$ containing $p$.
We say that the scheme $X$ is locally a complete intersection at $p \in X$ if there is an open neighbourhood $U$ of $p$ in $X$ such that, as an open subscheme of $X$, the scheme $U$ is isomorphic to a closed subscheme of some smooth scheme $V$ of finite type over $F$ defined by $\codim _V (U)$ equations.
\section{Preliminaries}\label{section2}

We start by recalling some well-known facts on degeneracy loci of morphisms of locally free sheaves.

Let $M(m,n)\cong \mathbb{A}^{mn}_F$ be the variety of $m \times n$-matrices over $F$ and for $0 \leq k \leq \min \{m,n\}$ let $M_k(m,n)$ be the locus of matrices of rank at most $k$.
As a reduced closed subscheme of $M(m,n)$ it is defined by the minors of order $k+1$ (see e.g. \cite{ref9}, Chapter II, Proposition p. 71; although in \cite{ref9} one takes $F=\mathbb{C}$, the results of Chapter II do hold for arbitrary $F$).
The smooth locus $\Sm (M_k(m,n))$ of $M_k(m,n)$ is equal to $M_k(m,n)\setminus M_{k-1}(m,n)$ (see \cite {ref9}, Chapter II, Proposition p. 69).
The codimension of $M_k(m,n)$ in $M(m,n)$ is equal to $(m-k)(n-k)$.
In particular, for $A \in M_k(m,n) \setminus M_{k-1}(m,n)$, locally at $A$ the ideal of $M_k(m,n)$ in $M(m,n)$ is generated by $(m-k)(n-k)$ equations.

Let $X$ be a scheme of finite type defined over $F$ and let $\phi : E \rightarrow F$ be a homomorphism of locally free sheaves on $X$ with $\rk (E)=n$ and $\rk (F)=m$.
Let $U$ be an open subset of $X$ such that $E \vert _U$ and $F \vert _U$ are trivial.
Using fixed trivialisations on $U$ the morphism $\phi$ corresponds to a morphism $f_U :  U \rightarrow M(m,n)$.
For $0 \leq k \leq \min \{ m,n \}$ we obtain a closed subscheme $f_U^{-1}(M_k(m,n))$ of $U$.
Those closed subschemes patch together to obtain the so-called $k$-th degeneracy locus $X_k(\phi)$ of $\phi$.
As a set we have
\[
X_k (\phi) = \{ p \in X : \rk (\phi (p))\leq k \} \text { .}
\]
In case $p \in X_k(\phi)$ and $\dim _p(X)=d$ then $\dim _p(X_k(\phi)) \geq d-(m-k)(n-k)$ and in case $p \notin X_{k-1}(\phi)$ then locally at $p$ the subscheme $X_k(\phi)$ of $X$ is defined by $(m-k)(n-k)$ equations.
In case $X$ is irreducible at $p$ and $\dim _p(X_k(\phi))=d-(m-k)(n-k)$ then we say $X_k(\phi)$ has the expected codimension in $X$ at $p$.
It is know that in that case, if $X$ is smooth at $p$, then $X_k(\phi)$ is Cohen-Macauly at $p$ (also if $p \in X_{k-1}(\phi)$) (see \cite{ref9}, Chapter II, Proposition (4.1)).

\begin{remark}\label{rem1.1}
We are going to obtain the following situation.
The scheme $X$ is irreducible at $p$ of dimension $d$ and it is locally a complete intersection at $p$.
On $X$ there is a homomorphism of locally free sheaves $\phi : E \rightarrow F$ with $\rk (E)=n$ and $\rk (F)=m$ and $p \in X_k(\phi)$.
From the previously mentioned facts it follows that, in case $\dim _p(X_k(\phi))=d-(m-k)(n-k)$ then also $X_k(\phi)$ is locally a complete intersection at $p$.
It is known that this implies $X_k(\phi)$ is Cohen-Macauly at $p$ (see Proposition 18.13 in \cite{ref5}).
\end{remark}

Let $C$ be a smooth irreducible projective curve of genus $g$ over $F$.
We recall the scheme structure on $W^r_d(C)$ as described in \cite {ref9}, Chapter IV, Remark (3.2).
We will use the following notations throughout in this paper.

\begin{notation}\label{not1.1}
We write $\mathcal{P}_d$ to denote the Poincar\'e bundle on $C \otimes \Pic ^d(C)$ and we write $p_d : C \times \Pic ^d(C) \rightarrow C$ and $\pi _d : C \times \Pic ^d(C) \rightarrow \Pic ^d(C)$ to denote the projections (and often we just omit the subscript $d$).
\end{notation}

Let $D$ be a divisor of large degree on $C$ and consider the exacte sequence
\[
0 \rightarrow \mathcal{P}_d \rightarrow \mathcal {P}_d \otimes p^*(\mathcal{O}_C(D)) \rightarrow \mathcal {P}_d \otimes p^*(\mathcal{O}_C(D))\vert _{D \times \Pic^d(C)} \rightarrow 0 \text { .}
\]
Since for each $L \in \Pic^d(C)$ we have $h^1(L(D))=0$ we have an exact sequence
\begin{multline*}
0 \rightarrow \pi_*(\mathcal{P}_d) \rightarrow \pi_*(\mathcal{P}_d \otimes (\mathcal{O}_C(D))) \rightarrow \\  \pi_*(\mathcal{P}_d \otimes p^*(\mathcal {O}_C(D))\vert _{D\times \Pic^d(C)}) \rightarrow R^1\pi_*(\mathcal {P}_d) \rightarrow 0 \text { .}
\end{multline*}
Using Grauert's Theorem (see e.g. \cite {ref1}, Chapter III, Corollary 12.9) we have $E=\pi_*(\mathcal{P}_d\otimes p^*(\mathcal{O}_C(D))$ is locally free of rank $d+\deg (D)-g+1$ and $F=\pi_*(\mathcal{P}_d \otimes p^*(\mathcal{O}_C(D))\vert _{D \times \Pic^d(C)})$ is locally free of rank $\deg (D)$, hence we obtain a so-called locally free presentation of $R^1\pi_*(\mathcal{P}_d)$, namely the exact sequence
\[
E \xrightarrow {\phi} F \rightarrow R^1\pi_*(\mathcal{P}_d) \rightarrow 0 \text { .}
\]
For $L \in \Pic^d(C)$ we have $R^1\pi_*(\mathcal{P})(L) \equiv H^1(C,L)$ (this follows from Cohomology and base change, see e.g. \cite{ref1}, Chapter III, Theorem 12.11).
From the Riemann-Roch Theorem it therefore follows that as a set (using the notation introduced above)
\[
W^r_d(C)=\Pic^d(C)_{r-d+g}(\phi) \text { .}
\]
The ideal of the scheme structure on $\Pic^d(C)_{r-d+g}(\phi)$ mentioned above is the so-called $(r-d+g)$-th Fitting ideal of $R^1\pi_*(\mathcal{P}_d)$ and this ideal $\Fitt _{r-d+g}(R^1\pi_*(\mathcal{P}_d))$ is independent of the chosen representation of $R^1\pi_* (\mathcal{P}_d)$ (locally, see e.g. \cite {ref5} Corollary-Definition 20.4 and globally see \cite{ref5} Corollary 20.5, showing all local definitions patch together; the theory of Fitting ideals originally appeared in \cite {ref14}).

\begin{definition}\label{schemespecialdivisors}
Let $C$ be a smooth curve of genus $g$.
The scheme $W^r_d(C)$ of special invertible sheaves is the closed subscheme of $\Pic^d(C)$ defined by the sheaf of ideals $\Fitt _{r-d+g}(R^1\pi_*(\mathcal{P}_d))$.
\end{definition}

Now let $f:  C \rightarrow \mathbb{P}^1$ be a morphism of degree $k$ corresponding to a line bundle $M$ on $C$ (and we always assume $h^0(M)=2$).
For $L \in \Pic^d(C)$ in \cite{ref7} and \cite{CPJ} one intoduces the so-called spitting sequence $\overrightarrow {e}$ of $L$ (with respect to $f$).
Associated to it one defines Brill-Noether splitting loci and degeneracy loci.
We recall the definition and we recall the scheme structure on degeneracy loci.

\begin{definition}\label{def1.1}
Let $L$ be a line bundle of degree $d$ on $C$ and write $f_*(L) \cong \mathcal{O}_{\mathbb{P}^1}(e_1) \oplus \cdots \oplus \mathcal{O}_{\mathbb{P}^1}(e_k)$ for some integers $e_1 \leq \cdots \leq e_k$.
We call $\overrightarrow {e}=(e_1, \cdots , e_k)$ the splitting type of $L$ (and then $f_*(L)$ is denoted by $\mathcal{O}(\overrightarrow {e})$).
\end{definition}

The splitting type in Definition \ref{def1.1} satisfies
\begin{equation}\label{eq1}
\Sigma_{i=1}^k e_i=d-g+1-k \text { .}
\end{equation}
From now on, fixing $d$, $g$ and $k$, a sequence $(e_1 , \cdots , e_k)$ satisfying Equation \ref{eq1} is called a splitting sequence.

\begin{definition}\label{def1.2}
Associated to a splitting type $\overrightarrow {e}$ we define the Brill-Noether splitting locus (as a set)
\[
\Sigma _{\overrightarrow {e}}(C,f)=\{ L \in \Pic^d(C) : f_*(L) \cong \mathcal{O}(\overrightarrow {e}) \} \text { .}
\]
\end{definition}

We introduce a partial ordening on the splitting sequences considered in Definition \ref{def1.1}.
\begin{definition}\label{def1.3}
Let $\overrightarrow {e}$ and $\overrightarrow {e'}$ be two splitting sequences of length $k$ with $\Sigma_{i=1}^k e_i=\Sigma_{i=1}^k e'_i$ then $\overrightarrow {e'} \leq \overrightarrow {e}$ if and only if for all $1 \leq l \leq k$ we have $\Sigma_{i=1}^l e'_i \leq \Sigma_{i=1}^l e_i$.
If moreover $\overrightarrow{e} \neq \overrightarrow {e'}$ then we write $\overrightarrow {e'}<\overrightarrow {e}$.
\end{definition}

\begin{definition}\label{def1.4}
Let $\overrightarrow {e}$ be a splitting sequence.
The associated Brill-Noether degeneracy locus of $f$ is defined by (as a set)
\[
\overline{\Sigma}_{\overrightarrow {e}}(C,f)= \cup_{\overrightarrow {e'} \leq \overrightarrow {e}}\Sigma_{\overrightarrow {e'}}(C,f) \text { .}
\]
It is a closed subspace of $Pic^d(C)$.
\end{definition}

In \cite {ref6} one introduces a scheme structure on those Brill-Noether degeneracy loci.
By definition we have
\[
L\in \overline{\Sigma}_{\overrightarrow {e}}(C,f) \Leftrightarrow \text { for each integer $n$ }: h^0(L+nM)\geq \sum_{j=1}^k(e_j+n+1 : e_j+n \geq 0) \text { .}
\]
\begin{notation}\label{notrn}
Once $\overrightarrow {e}$ is fixed we write $r(n)$ to denote $\sum_{j=1}^k (e_j+n+1 : e_j+n \geq 0)-1$.
\end{notation}
Then the previous equivalence can be written as an equation of sets
\[
\overline {\Sigma}_{\overrightarrow{e}}(C,f)=\cap_{n \in \mathbb{Z}}W^{r(n)}_{d+nk}(C) \text { .}
\]
In order to make this intersection meaningful we have to identify $W^{r(n)}_{d+nk}(C) \subset \Pic^{k+nk}(C)$ with its image under the isomorphism $\Pic ^{d+nk}(C) \rightarrow \Pic^d(C):L \rightarrow L-nM$.
Of course in this way we can consider $W^{r(n)}_{d+nk}(C)$ as a closed subscheme of $\Pic^d(C)$ using the scheme structure introduced in Definition \ref{schemespecialdivisors}.

\begin{notation}\label{notFitt}
The subschemes $W^{r(n)}_{d+kn}(C)$ of $\Pic ^{d+nk}(C)$ are defined by the $(g-(d+kn)+r(n))$-th Fitting ideal of $R^1\pi_*(\mathcal{P}_d \otimes p^*(M^{\otimes n}))$ for all $n \in \mathbb{Z}$.
The corresponding ideals of $W^{r(n)}_{d+nk}(C) \subset \Pic^d(C)$ are denoted by $\Fitt (\overrightarrow {e},n)$.
\end{notation}

\begin{definition}\label{schemedegloci}
Let $C$ be a smooth curve of genus $g$ and let $f: C \rightarrow \mathbb{P}^1$ ba a morphism of degree $k$ corresponding to an invertible sheaf $M$ on $C$ with $h^0(C,M)=2$.
Fix a non negative integer $d$ and let $\overrightarrow {e}$ be a splitting sequence.
The Brill-Noether degeneracy scheme $\overline{\Sigma}_{\overrightarrow {e}}(C,f)$ is the scheme theoretical intersection $\cap_{n \in \mathbb{Z}}W^{r(n)}_{d+nk}(C)$, so it is the closed subscheme of $\Pic^d(C)$ defined by the sheaf of ideals generated by all sheafs of ideals $\Fitt (\overrightarrow{e},n)$.
\end{definition}

In Section \ref{section2} for certain elements $L \in \overline {\Sigma}_{\overrightarrow {e}}(C,f)$ we are going to restrict those numbers $n$ needed to describe $\overline{\Sigma}_{\overrightarrow {e}}(C,f)$ locally at $L$. At those points $L$ we are going to describe the Brill-Noether degeneracy scheme as a sequence of degeneracy loci such that in case $f$ is general their codimensions are the expected ones.  As an application we obtain some smoothness results on $\overline{\Sigma}_{\overrightarrow {e}}(C,f)$ in case $f$ is general.

\begin{notation}\label{notue}
Associated to a splitting sequence $\overrightarrow{e}$ we define the integer $u(\overrightarrow {e})=\sum _{1 \leq i <j \leq k} \max \{ 0,e_j-e_i-1 \}$.
\end{notation}

Now we recall some results proved on those Brill-Noether degeneracy loci in case $C$ is a general $k$-gonal curve.
\begin{proposition}\label{prop1.1}
Let $C$ be a general $k$-gonal curve of genus $g$ (as always in this paper we assume $k< [\frac{g+3}{2}]$).
We write $f : C \rightarrow \mathbb{P}^1$ for a morphism of degree $k$ corresponding to the unique $M \in W^1_k(C)$.
\begin{enumerate}
\item $\Sigma_{\overrightarrow{e}}(C,f)$ is empty if $u(\overrightarrow {e})>g$ and in case it is non-empty it has dimension $g-u(\overrightarrow {e})$ (see \cite{ref7}, \cite{CPJ} and \cite{FFR} for different proofs).
\item In case $u(\overrightarrow {e})\leq g$ then $\Sigma _{\overrightarrow {e}}(C,f)$ is not empty (see \cite{ref7}, \cite{CPJLLV} and \cite{FFR} for different proofs).
\item $\overline {\Sigma}_{\overrightarrow {e}}(C,f)$ is the closure of $\Sigma_{\overrightarrow {e}}(C,f)$ (see \cite{ref7}).
\item In case $u(\overrightarrow {e})<g$ then the scheme $\overline{\Sigma}_{\overrightarrow {e}}(C,f)$ is integral (see \cite{ref6}).
\item $\Sigma _{\overrightarrow {e}}(C,f)$ is smooth as a set (see \cite {ref7}) and it is contained in the smooth locus of the scheme $\overline {\Sigma}_{\overrightarrow {e}}(C,f)$ (see \cite {ref2}, remark 5, where it is explained this follows immediately from results of H. Larson).
\end{enumerate}
\end{proposition}

In \cite{ref7} one also obtains a description of the irreducible components of $W^r_d(C)$ based on some of the results of Proposition \ref{prop1.1}.
\begin{definition}\label{balplusbal}
A splitting sequence $\overrightarrow {e}$ is called balanced plus balanced if the following holds.
Let $1 \leq c \leq k$ with $e_i <0$ for $i<c$ and $e_i \geq 0$ for $i \geq c$.
Then $e_{c-1} - e_1 \leq 1$ and $e_k - e_c \leq 1$.
\end{definition}

In \cite{ref2} we introduced the following notation to describe a splitting sequence of balanced plus balanced type.
We say $\overrightarrow {e}$ is of type $B(a,b,y,u,v)$ in case $\overrightarrow {e}$ is of the form
\[
(\overbrace{-b-1, \cdots , -b-1} ^x, \overbrace{-b, \cdots, -b}^y, \overbrace {a, \cdots, a}^u, \overbrace {a+1, \cdots , a+1}^v)
\]
with $b \geq 1$, $a \geq 0$, $y>0$, $u>0$, $v \geq 0$ and then of course $x=k-y-u-v \geq 0$.
The irreducible components of $W^r_d(C)$ are the spaces $\overline {\Sigma}_{\overrightarrow {e}}(C,f)$ with $\overrightarrow {e}$ of type $B(a,b,y,u,v)$ satisfying Equation \ref{eq1}, $u(\overrightarrow {e})\leq g$, $u(a+1)+v(a+2)=r+1$ and $b \geq 2$ or $(a,b,v)=(0,1,0)$ (of course, in case $u(\overrightarrow {e})=g$ then it is a finite set and each point of it is an irreducible component).

\begin{remark}\label{types}
Also  in \cite {ref2} we called such an irreducible component of type I if $(a,v)=(0,0)$, of type II in case $a=0$ and $v \geq 1$ and of type III otherwise (see \cite{ref2} for the motivation of making this partition).
\end{remark}

\begin{notation}\label{notfree}
In \cite{ref2} we defined $\L \in \Pic^d(C)$ being free from $M$ in case $h^0(L-M)=0$ and $\Pic^d_f(C) \subset \Pic^d(C)$ 
being the open subset of invertible sheaves $L$ free from $M$.
We notice that  for a general $k$-gonal curve $C$ those subsets $\Pic^d_f(C)$ satisfies the classical statements of Brill-Noether theory concerning dimension and non-emptiness in case $r \leq k-1$.
The irreducible components of $W^r_{d,f}(C)=W^r_d(C) \cap \Pic^d_f(C)$ are the intersections of $\Pic^d_f(C)$ with the irreducible components of $W^r_d(C)$ of type I.
\end{notation}

\section{The ideal of the splitting degeneracy loci}\label{section2}
We first introduce some notations related to a splitting sequence $\overrightarrow{e}$.

\begin{notation}\label{not2.1}
Let $e \in \mathbb{Z}$ and define $i(\overrightarrow {e},e)=k+1$ if $e_k <e$ and $i(\overrightarrow{e},e)= \min \{ i\in \mathbb{Z} : e \leq e_i\}$ if $e\leq e_k$.
We define the splitting type $\overrightarrow{e}[e] = (e_1[e], \cdots ,e_k[e])$ with $\Sigma_{i=1}^k e_i[e]=d-g+1-k$ such that $e_j[e]=e_j$ if $j \geq i(\overrightarrow{e},e)$ (in particular $\overrightarrow{e}[e]=\overrightarrow {e}$ in case $e \leq e_1$) and in case $e_1 <e$ then $e_{i(\overrightarrow{e},e)-1}[e]-e_1[e] \leq 1$.
This last condition means $\overrightarrow{e} [e]$ is completed in a balanced way (see Definition \ref{balplusbal}) for $i< i(\overrightarrow{e},e)$.
It also implies that $\overrightarrow {e} [e]=\overrightarrow {e}$ if and only if $e \leq \min \{e_i : e_i \geq e_1+2 \}$.
\end{notation}

\begin{example}\label{ex3.1}
Assume $g=20$ and consider $\overrightarrow {e}=(-4, -3, -1, -1, 1, 3)$.
Then $k=6$ and $d=20$.
In case $e>3$ we have $i(\overrightarrow {e},e)=7$ and $\overrightarrow {e}[e]=(-1,-1,-1,-1,-1,0)$.
In case $e \in \{2,3\}$ we have $i(\overrightarrow {e},e)=6$ and $\overrightarrow {e}[2]=\overrightarrow {e}[3]=(-2,-2,-2,-1,-1,3)$.
In case $e \in \{ 0,1 \}$ we have $i(\overrightarrow {e},e)=5$ and $\overrightarrow {e}[0]=\overrightarrow {e}[1]=(-3, -2, -2, -2, 1, 3)$.
In case $e \in \{-2, -1 \}$ we have $i(\overrightarrow {e},e)=3$ and $\overrightarrow {e}[-2]=\overrightarrow {e}[-1]=(-4, -3, -1, -1, 1,3)$.
In case $e =-3$ we have $i(\overrightarrow {e},e)=2$, in case $e \leq -4$ we have $i(\overrightarrow {e},e)=1$ and in those cases we have $\overrightarrow {e}[e]=\overrightarrow {e}$.
\end{example}

\begin{definition}\label{def2.1}
We also introduce the following loci in $\Pic^d(C)$:

\[
\overset{\circ}{ \Sigma}_{\overrightarrow{e}[e]}(C,f)=\cup_{\overrightarrow{e'}}\Sigma_{\overrightarrow{e'}}(C,f)\subset \overline{\Sigma}_{\overrightarrow{e}[e]}(C,f)
\]
where the union is taken for all $\overrightarrow{e'} \leq \overrightarrow{e}[e]$ satisfying $e'_j=e_j$ for all $j \geq i(\overrightarrow{e},e)$ and $e'_j<e$ for $j < i(\overrightarrow {e},e)$.
\end{definition}

\begin{example}\label{ex3.2}[Continuation of Example \ref{ex3.1}]
Let $g=20$, $\overrightarrow {e}=(-4,-3,-1,-1,1,3)$ and take $e=1$.
Let $\overrightarrow{e}'_1=(-4,-2,-2,-1,1,3)$, $\overrightarrow {e}'_2=(-3,-3,-3,0,1,3)$, $\overrightarrow {e}'_3=(-3,-2,-2,-2,0,4)$, $\overrightarrow {e}'_4=(-4,-2,-2,-2,2,3)$ and $\overrightarrow {e}'_5=(-4,-3,-3,1,1,3)$.
For all $1 \leq i \leq 5$ we have $\overrightarrow {e}'_i \leq \overrightarrow {e}[1]$.
Since $(e'_1)_5=(e'_2)_5=e_5$, $(e'_1)_6=(e'_2)_6=e_6$ and $(e'_1)_4 <1$, $(e'_2)_4<1$ we have $\Sigma_{\overrightarrow {e}_i}(C,f) \subset \overset { \circ}{\Sigma}_{\overrightarrow {e}[1]}(C,f)$ for $i \in \{ 1,2 \}$.
Since $(e'_3)_5 \neq e_5$, $(e'_3)_6 \neq e_6$ and $(e'_4)_5 \neq e_5$ we have $\Sigma _{\overrightarrow {e}_i}(C,f) \nsubseteq \overset {\circ}{\Sigma}_{\overrightarrow {e}[1]}(C,f)$ for $3 \leq i \leq 4$.
Although $(e'_5)_5=e_5$ and $(e'_5)_6=e_6$ we have $(e'_5)_4=1$ and therefore also $\Sigma _{\overrightarrow{e}'_5}(C,f) \nsubseteq \overset {\circ}{\sigma}_{\overrightarrow {e}[1]}(C,f)$. 
\end{example}

\begin{remark}\label{warning}
In Example \ref{ex3.2} although $\overrightarrow {e}[0]=\overrightarrow {e}[1]$ we have $\Sigma _{\overrightarrow {e}'_2}(C,f) \subset \overset {\circ}{\Sigma}_{\overrightarrow {e}[1]}(C,f)$ but $\Sigma_{\overrightarrow{e}'_2}(C,f) \nsubseteq \overset {\circ}{\Sigma}_{\overrightarrow {e}[0]}(C,f)$, hence $\overset {\circ}{\Sigma}_{\overrightarrow {e}[1]}(C,f) \neq \overset {\circ}{\Sigma}_{\overrightarrow {e}[0]}(C,f)$.
Therefore in the expression $\overset {\circ}{\Sigma}_{\overrightarrow {e}[e]}(C,f)$ the value of $e$ gives more information than just indicating $\overrightarrow {e}[e]$.
Therefore for a splitting sequence $\overrightarrow {e}$ the expression $\overset {\circ}{\Sigma}_{\overrightarrow {e}}(C,f)$ has no meaning.
\end{remark}

\begin{remark}\label{rem2.1}
As it follows from the proof of Lemma 8 in \cite{ref2}, in case $n \geq -(e_1+1)$ then $h^0(C,L+nM) \geq r(n)+1$ is equivalent to $h^1(C,L+nM)=0$, in particular $\Fitt(\overrightarrow {e},n)=0$.
Also from Lemma 8 in \cite{ref2} it follows $\Fitt (\overrightarrow {e},n)$ is trivial at $L$ in case $n<-e_k$.
\end{remark}

\begin{proposition}\label{prop1}
Let $(C,f)$ be a general $k$-gonal curve of genus $g$.
Let $E = \min \{e_i : e_i \geq e_1+2 \}$.
Locally at $L \in \overset{\circ}{ \Sigma}_{\overrightarrow{e}[E]}(C,f) \subset \overline {\Sigma}_{\overrightarrow {e}}(C,f)$ the scheme $\overline {\Sigma}_{\overrightarrow {e}}(C,f)$ is defined by the Fitting ideals $\Fitt (\overrightarrow {e},n)$ with $n \leq -E$ and $-n \in  \{ e_2, \cdots , e_k \}$.
In particular, locally at $L$ the Fitting ideals $\Fitt (\overrightarrow {e},n)$ with $n>-E$ are contained in the sum of the other ones.
Also locally at $L$ the scheme $\overline {\Sigma}_{\overrightarrow {e}}(C,f)$ is a locally complete intersection in $\Pic^d(C)$.
\end{proposition}

\begin{remark}\label{remV}
Assume $\overrightarrow{e}$ is of type $B(e,b,y,m_1,m_2)$ with $\Sigma _{i=1}^r e_i=d-g+1-k$ (then $b$ and $y$ are determined by the fact that $\overrightarrow {e}$ is of balanced plus balanced type).
Then $\overset{\circ}{\Sigma}_{\overrightarrow {e}[0]}(C,f)$ as a set is the same as the set introduced in \cite{FFR} denoted by $V^r_{d,l}(C,M)$ (in this paper we are going to write $V^r_{d,l}(C,f)$) with $l=em_1+(e+1)m_2$ and $r+1=(e+1)m_1+(e+2)m_2$.
See Definition \ref{defV} for a concrete description of $V^r_{d,l}(C,f)$ as a set given in \cite{FFR}.
Also notice that we can restrict to the cases with $e_1 \leq -2$ because otherwise $L \in \overset{\circ}{\Sigma}_{\overrightarrow {e}[0]}(C,f)$ is not special.
In case $C$ is a general $k$-gonal curve of genus $g$ then the closures of those subsets are related to the irreducible components of $W^r_d(C)$ (see \cite{FFR}).
In the present paper, since $V^r_{d,l}(C,f)$ is an open subset of $\overline{\Sigma}_{\overrightarrow{e}}(C,f)$ we consider it as a scheme with the natural scheme structure described in Definition \ref{schemedegloci}.
It is using this scheme structure that it fits in families of coverings of genus $g$ and degree $d$ on $\mathbb{P}^1$.

In \cite{FFR} one studies curves $C$ on so called degree $k$ elliptic K3-surfaces.
Such K3-surface has a pencil of elliptic curves inducing a covering $f:C \rightarrow \mathbb{P}^1$.
Using moduli spaces of Bridgeland stable sheaves on the K3 surface in \cite{FFR} one introduces a scheme structure on $V^r_{d,l}(C,f)$ and they prove this is a smooth scheme.
However the relation of that scheme structure with the scheme structure from Definition \ref{defV} is not considered in \cite{FFR}.
Therefore although they obtain $V^r_{d,l}(C,f)$ is smooth as a set for a general such rather special $k$-gonal curve, it is not even clear from the arguments in \cite{FFR} that $V^r_{d,l}(C,f)$ is smooth as a set in case $C$ is a general $k$-gonal curve.
It should also be noted that the arguments in \cite{FFR} make use of characteristic 0 arguments.
Theorem C settles this problem in arbitrary characteristic.
\end{remark}

We summarize the following definition occuring in Remark \ref{remV} using the notations from Remark \ref{remV}.
\begin{definition}\label{defV}
Let
\begin{multline*}
V^r_{d,l}= \{L \in \Pic^d(C) : h^0(L)=r+1, h^0(L-(e+1)M)=m_2 \\ \text { and } h^0(L-eM)=m_1+2m_2 \} \text { .}
\end{multline*}
It is an open subset of $\overline{\Sigma}_{\overrightarrow {e}}(C,f)$ with $\overrightarrow {e}$ of balanced plus balanced type and we define the scheme $V^r_{d,l}(C,f)$ as the open subscheme of the scheme $\overline{\Sigma}_{\overrightarrow {e}}(C,f)$ defined in Definition \ref{schemedegloci}.
\end{definition}

The meaning of Proposition \ref{prop1} is the following: if $L \in \overset{\circ}{\Sigma}_{\overrightarrow {e}[E]}$ and $\overrightarrow {e'}$ is the splitting sequence of $L$ then the values $e'_i <E$ have no influence on the Fitting ideals needed to define the scheme $\overline{\Sigma}_{\overrightarrow {e}}(C,f)$ locally at $L$.
The proof of the proposition makes use of an extension of the base point free pencil trick.
This base point free pencil trick is the following lemma in case $m=1$ and we refer to \cite{ref9} for its proof in that case.
This extension implies and is very similar to Proposition 2.2 in \cite{FFR} but we need a more detailled description compared to what is stated in \cite{FFR}.

\begin{lemma}\label{lemma1}
Let $M$ be a line bundle on a smooth curve $C$ such that $h^0(M)=2$ and the linear system $\vert M \vert$ is base point free.
Let $L$ be another line bundle on $C$ with $h^0(L)\geq 1$ and let $m \in \mathbb{Z}_{\geq 1}$.
Consider the natural multiplication map
\[
\phi _m : H^0(C,L) \otimes S^m H^0(C,M) \rightarrow H^0 (C, L+mM)
\]
(here $S^m$ denotes the symmetric product).
The kernel of $\phi _m$ is naturally isomorphic to $H^0(C,L-M)^{\oplus m}$.
More concretely, after choosing a base $s,t$ for $H^0(C,M)$, the element of $\ker (\phi _m)$ corresponding to $(z_1, \cdots, z_m) \in H^0(C,L-M)^{\oplus m}$ is equal to
\begin{multline*}
z_1t \otimes s^m +(z_2t-z_1s) \otimes s^{m-1}t + (z_3t-z_2s)\otimes s^{m-2}t^2 + \cdots\\
\cdots + (z_mt-z_{m-1}s) \otimes st^{m-1}-z_ms\otimes t^m \text { .}
\end{multline*}
\end{lemma}

\begin{proof}
It is clear that the explicit description of elements of $\ker (\phi_m)$ identifies $H^0(C,L-M)^{\oplus m}$ with a subvectorspace of $\ker (\phi _m)$.

Let $s$, $t$ be a base for $H^0(C,M)$ and let $Z(s)$ and $Z(t)$ be the corresponding effective divisors on $C$.
By assumption $Z(s) \cap Z(t) = \emptyset$.
We already mentioned the case $m=1$ is the classical base point free pencil trick.
Then one has the explicite description
\[
\ker (\phi _1) = \{zt\otimes s-zs \otimes t : z \in H^0(C,L-M) \} \text { .}
\]

Let $m \geq 2$ and as an induction hypothesis we assume
\begin{multline*}
\ker (\phi _{m-1}) = \{ z_1t\otimes s^{m-1}+(z_2t-z_1s)\otimes s^{m-2}t +(z_3t-z_2s)\otimes s^{m-3}t^2+ \cdots \\
+(z_{m-1}t-z_{m-2}s) \otimes st^{m-2}-z_{m-1}s\otimes t^{m-1} : z_1, \cdots , z_{m-1} \in H^0(C,L-M) \} \text { .}
\end{multline*}
Let $\Sigma _{i=0}^m y_i \otimes s^{m-i}t^i \in \ker (\phi _m)$.

First assume $y_0=0$.
In $H^0(C,L+mM)$ we have $(\Sigma_{i=1}^my_is^{m-i}t^{i-1})\cdot t=0$ hence $\Sigma _{i=1}^my_is^{m-i}t^{i-1}=0$ in $H^0(C,L+(m-1)M)$.
In particular $\Sigma _{i=1}^my_i\otimes s^{m-i}t^{i-1} \in \ker (\phi_{m-1})$ hence there exist $z_2, \cdots , z_m \in H^0(C,L-M)$ with
\begin{multline*}
\Sigma _{i=1}^my_i\otimes s^{m-i}t^{i-1}=z_2t\otimes s^{m-1}+(z_3t-z_2s)\otimes s^{m-2}t + \cdots \\
+ (z_mt-z_{m-1}s)\otimes st^{m-2}-z_ms\otimes t^{m-1} \text { .}
\end{multline*}
Multiplying with $t$, adding a term $0 \cdot t \otimes s^m$ on the right and replacing $z_2t\otimes s^{m-1}t$ by $(z_2t-0\cdot s)\otimes s^{m-1}t$ we obtain $(0,z_2, \cdots , z_m) \in H^0(C,L-M)^{\oplus m}$ corresponding to $\Sigma _{i=0}^m y_i \otimes s^{m-i}t^i$ as described in the statement.

Now assume $y_0 \neq 0$.
In $H^0(C,L+mM)$ we have $s^my_0+(\Sigma _{i=1}^my_is^{m-i}t^{i-1})\cdot t=0$ with $\Sigma _{i=1}^my_is^{m-i}t^{i-1} \in H^0(C,L+(m-1)M)$.
Since $Z(s) \cap Z(t) = \emptyset$ it implies there exists $z_1 \in H^0(C,L-M)$ with $s^my_0 = s^mtz_1$ and we have
\[
(sz_1+y_1)s^{m-1}t+\Sigma _{i=2}^m y_is^{m-i}t^i=0
\]
in $H^0(C,L+mM)$.
So we find
\[
(sz_1+y_1)\otimes s^{m-1}t + y_2 \otimes s^{m-2}t^2 + \cdots + y_m \otimes t^m \in \ker (\phi _m)
\]
having its term in $H^0(C,L) \otimes s^m$ being zero.
So we already proved the existence of $z_2, \cdots , z_m \in H^0(C,L-M)$ with $s_1z_1+y_1=z_2t, y_2=z_3t-z_2s, \cdots , y_{m-1}=z_mt-z_{m-1}s, y_m=-z_ms$.
So we find $y_1=z_2t-s_1z_1$, while we already know that $y_0=z_1t$.
So we find $\Sigma _{i=0}^m y_i \otimes s^{m-i}t^i$ corresponds to $(z_1, \cdots , z_m) \in H^0(C,L-M)^{\oplus m}$.

This finishes the proof of the lemma.
\end{proof}

\begin{proof}[Proof of Proposition \ref{prop1}]
Let $U$ be  the subscheme of  $\Pic^d(C)$ defined by the ideal sheaf generated by all $\Fitt (\overrightarrow {e},n)$ with $n \geq -E$ and $-n = e_i$ for some $1 \leq i \leq k$.
We are going to prove the following facts (notice that $\overrightarrow {e}[E]=\overrightarrow {e}$)
\begin{itemize}
\item locally at $L \in \overset {\circ}{\Sigma}_{\overrightarrow {e}[E]}(C,f)$ the supports of $U$ and $\overline {\Sigma }_{\overrightarrow {e}[E]}(C,f)$ are equal,
\item locally at $L \in \overset {\circ}{\Sigma}_{\overrightarrow {e}[E]}(C,f)$ the scheme $U$ is locally a complete intersection,
\item locally at $L \in \Sigma_{\overrightarrow {e}[E]}(C,f)$ the schemes $\overline {\Sigma }_{\overrightarrow {e}[E]}(C,f)$ and $U$ are equal.
\end{itemize}

From the second statement we find $U$ is Cohen-Macauly at $L \in \overset {\circ}{\Sigma}_{\overrightarrow {e}[E]}(C,f)$.
From the third statement, using Proposition \ref{prop1.1} (5), we find $U$ is smooth at $L \in \Sigma_{\overrightarrow {e}[E]}(C,f)$.
Moreover from the first statement it follows $\Sigma_{\overrightarrow {e}[E]}(C,f)$ is an open subset of $U$ and from Proposition \ref{prop1.1} (4) we know $U$ is irreducible at $L \in \overset {\circ}{\Sigma}_{\overrightarrow {e}[E]}(C,f)$.
Then the unmixedness theorem of Cohen-Macauly schemes implies $U$ is a reduced scheme locally at $L \in \overset {\circ}{\Sigma}_{\overrightarrow {e}[E]}(C,f)$.
Then again from Proposition \ref{prop1.1} (4) and from the first statement we find locally at $L \in \overset {\circ}{\Sigma}_{\overrightarrow {e}[E]}(C,f)$ the schemes $U$ and $\overline {\Sigma }_{\overrightarrow {e}[E]}(C,f)$ are equal.
Therefore proving those facts on the scheme $U$ implies the statement of the proposition.
We used $u(\overrightarrow {e}) <g$ in those arguments but in case $u(\overrightarrow {e})=g$ we have $\overline{\Sigma}_{\overrightarrow {e}}(C,f)=\Sigma_{\overrightarrow {e}}(C,f)$ and we don't need to use the irreducibility statement.

Let $m=m(\overrightarrow {e})$ the number of  different values $e_i$ in $\overrightarrow {e}$ satisfying $E \leq e_i \leq e_k$.
We first prove the existence of $U$ in the case $m=1$.
Then in case $m>1$ we can use the following induction hypothesis.
For all splitting types $\overrightarrow {e'}$ with $m(\overrightarrow {e'})=m-1$ we define $E'$ for $\overrightarrow {e'}$ in the same way as we defined $E$ for $\overrightarrow {e}$.
Then we assume that for such splitting type $\overrightarrow {e'}$ at $L \in \overset {\circ}{\Sigma}_{\overrightarrow {e'}[E']}$ the scheme $\overline {\Sigma}_{\overrightarrow {e'}[E']}(C,f)$ is locally defined by $\Fitt (\overrightarrow {e'},n)$ for $n \leq -E'$ and $n =e'_i$ for some $E' \leq i \leq k$  and it is locally a complete intersection.
Taking $E'$ being the minimal value $e_i$ of $\overrightarrow {e}$ with $e_i>E$ we apply this assumption for $\overrightarrow {e'}=\overrightarrow {e}[E']$.

\underline{The case $m=1$.}

In this case $U$ is the subscheme defined by $\Fitt (\overrightarrow{e},-E)$, hence as a scheme it is $W^{r(-E)}_{d-Ek}+EM \subset \Pic ^d(C)$.
We have $L \in \overset {\circ}{\Sigma}_{\overrightarrow {e}[E]}(C,f)$ if and only if $L \in U$, $h^0(L-EM)=r(-E)+1$ and $h^0(L-(E+1)M)=0$.
Those two conditions are open conditions on $U$, so locally at $L$ the supports of $U$ and $\overline{\Sigma}_{\overrightarrow {e}[-E]}(C,f)$ are equal.
Since the Brill-Noether Theory on $\Pic(C)_f$ (see Notation \ref{notfree}) is the same as Brill-Noether Theory on a general curve of genus $g$ it follows $U$ is locally a complete intersection at $L \in \overset {\circ}{\Sigma}_{\overrightarrow {e}[E]}(C,f)$.
In \cite{ref2}, Corollary 1 it is proved $U$ is smooth at $L \in \Sigma_{\overrightarrow {e}[-E]}(C,f)$.
Therefore the three statements on $U$ are fulfilled proving the proposition in case $m=1$.

From now on we assume $m>1$ and we use the induction hypothesis.
We define $E'$ and $\overrightarrow {e'}$ as mentioned before such that $\overset {\circ}{\Sigma}_{\overrightarrow {e}[E']}=\overset {\circ}{\Sigma}_{\overrightarrow {e'}[E']}$ and we assume the statement of the proposition holds for $\overset {\circ}{\Sigma}_{\overrightarrow {e'}[E']}$.

Fix and integer $e$ with $E \leq e <E'$.

\underline{The presentation of $R^1\pi_*(\mathcal{P}_d \otimes p^*(-eM))$.}

On $\overset {\circ}{\Sigma}_{\overrightarrow {e}[E']}(C,f)$ we consider the exact sequence (using notations as in Notation \ref{not1.1})
\begin{multline*}
\pi_* (\mathcal{P}_d \otimes p^*(-eM+D)) \rightarrow \pi_* (\mathcal{P}_d \otimes p^* (-eM+D)\vert _{\overset {\circ}{\Sigma}_{\overrightarrow {e}[E']} \times D})\\
 \rightarrow R^1\pi_* (\mathcal {P}_d \otimes p^*(-eM)) \rightarrow 0 \text { .}
\end{multline*}
By definition of $\overset {\circ}{\Sigma}_{\overrightarrow {e}[E']}(C,f)$ we have that the value of $h^0(C,(\mathcal {P}_d \otimes p^*(-E'\cdot M)_L)$ is constant on $\overset {\circ}{\Sigma}_{\overrightarrow {e}[E')}(C,f)$.
Using Lemma \ref{lemma1} we are going to replace the first locally free sheaf in this presentation by some quotient locally free sheaf.
In using Lemma \ref{lemma1} we take $z=E'-e$.

Since $\overset {\circ}{\Sigma}_{\overrightarrow {E}[E']}(C,f)$ is an integral scheme (because of Proposition \ref{prop1.1} (4)) it follows from the theorem of Grauert (see e.g. \cite {ref1}, Chapter III, Corollary 12.9) that $\pi_*((\mathcal{P}_d \otimes p^*(-E'M))^{\oplus (z+1)})$ is a locally free sheaf on $\overset {\circ}{\Sigma}_{\overrightarrow {e}[E']}(C,f)$.
Because of the same reason also $\pi_*((\mathcal{P}_d \otimes p^*(-(E'+1)M))^{\oplus z})$ is a locally free sheaf on $\overset {\circ}{\Sigma}_{\overrightarrow {e}[E']}(C,f)$.
On $\overset {\circ}{\Sigma}_{\overrightarrow {e}[E']}(C,f) \times C$ we consider the map $(\mathcal{P}_d\otimes p^*(-E'M))^{\oplus (z+1)} \rightarrow \mathcal{P}_d \otimes p^*(-eM)$ locally defined by ($s$, $t$ being a base of $H^0(C,M)$)
\[
(x_1, \cdots , x_{z+1}) \rightarrow x_1 \otimes s^z+x_2 \otimes s^{z-1}t + \cdots + x_z \otimes st^{z-1}+x_{z+1} \otimes t^z \text { .}
\]
Then on $\overset {\circ}{\Sigma}_{\overrightarrow {e}[-E']}(C,f)$ it gives rise to a morphism of locally free sheaves $\pi_* ((\mathcal{P}_d \otimes p^*(-E'M))^{\oplus {z+1}}) \rightarrow \pi_* (\mathcal {P}_d\otimes p^*(-eM+D))$.
For $L \in \overset {\circ}{\Sigma}_{\overrightarrow {e}[E']}(C,f)$ the morphism on the fibers at $L$ is the composition
\begin{multline*}
H^0(C,L-E'M)^{\oplus (z+1)} \cong H^0(C,L-E'M)\otimes S^z(C,M) \rightarrow H^0(C,L-eM)\\
 \rightarrow H^0(C,L-eM+D) \text { .}
\end{multline*}
Because the second arrow in this composition is injective it follows from Lemma \ref{lemma1} that the kernel of this map is in a natural way isomorphic to $H^0(C,L-(E'+1)M)^{\oplus z}$.
So  we obtain a monomorphism between locally free sheaves $\pi_*((\mathcal{P}_d\otimes p^*(-(E'+1)M))^{\oplus z}) \rightarrow \pi_* ((\mathcal{P}_d \otimes p^*(-E'M))^{\oplus (z+1)})$ being the kernel of that homomorphism of locally free sheaves on $\overset {\circ}{\Sigma}_{\overrightarrow {e}[E']}(C,f)$ that is injective on the fibers.
In particular the cokernel $N'$ of this monomorphism is locally free (see e.g. \cite{ref2}, Lemma 4) of rank
\[
(z+1)(\Sigma_{e_i \geq E'} (e_i-E'+1))-z(\Sigma _{e_i \geq E'}(e_i-E'))=(z+1)(k-n+1)+\Sigma _{e_i \geq E'}(e_i-E')
\]
with $n=\min \{i : e_i=E' \}$.
We obtain a monomorphism of locally free sheaves $N' \rightarrow \pi _*(\mathcal {P}_d \otimes p^*(-eM+D))$ that is injective on the fibers.
The cokernel $N$ is again locally free on $\overset {\circ}{\Sigma}_{\overrightarrow {e}[E']}(C,f)$ of rank
\[
d-ek+\deg (D)-g+1-(z+1)(k-n+1)-\Sigma _{i\geq n} (e_i-E') \text { .}
\]
However by construction $N'$ is contained in the kernel of the morphism $\pi_*(\mathcal {P}_d \otimes p^*(-eM+D))\rightarrow \pi_*((\mathcal {P}_d \otimes p^*(-eM+D))\vert _{\overset{\circ}{\Sigma} _{\overrightarrow {e}[E']}(C,f)\times D})$.
Therefore on $\overset {\circ}{\Sigma}_{\overrightarrow {e}[E']}(C,f)$ we obtain the exact sequence
\[
N \xrightarrow{u} \pi_*((\mathcal{P}_d\otimes p^*(-eM+D))\vert _{\overset {\circ}{\Sigma}_{\overrightarrow {e}[E']}(C,f)\times D}) \rightarrow R^1\pi_*(\mathcal{P}_d\otimes p^*(-eM))\rightarrow 0 \text { .}
\]
This is the presentation of $R^1\pi_*(\mathcal{P}_d \otimes p^*(-eM))$ we are going to use.

\underline{The degeneracy scheme of $\phi$.}

On $\overset {\circ}{\Sigma}_{\overrightarrow {e}[E']}(C,f)$ we can use this presentation to define the Fitting ideal $\Fitt (\overrightarrow {e},-e)$.
Define $t \geq 0$ such that $e_{n-t}=e$ and $e_{n-t-1}<e$ (we put $t=0$ in case $e_j\neq e$ for all $j$; remember $n = \min \{i : e_i=E' \}$).
For $L \in \overset {\circ}{\Sigma}_{\overrightarrow {e}[E']}(C,f)$ we have $L \in \overset {\circ}{\Sigma} _{\overrightarrow {e}[E]}(C,f)$ if and only if for each $e$ satisfying $E \leq e < E'$ we have $h^0(L-eM)= \Sigma _{i\geq n} ((e_i-E')+E'-e+1)+t$.
From the Riemann-Roch Theorem we know $h^0(L-eM)\geq \Sigma _{i\geq n} ((e_i-E')+E'-e+1)+t$ if and only if $h^1(L-eM)\geq \Sigma _{i \geq n}((e_i-E')+E'-e+1)+t-d+ke+g-1$, equivalently 
\[
\rk (u(L)) \leq \deg (D) - \Sigma _{i \geq n}((e_i-E')+E'-e+1)-t+d-kE'+(E'-e)k-g+1 \text { .}
\]
This is a non-trivial condition if and only if $\rk (N)$ is larger than this lower bound on $\rk (u(L))$, i.e. if and only if $t>0$.
This already implies that the Fitting ideals $\Fitt (\overrightarrow {e},-e)$ are trivial when restricted to $\overset {\circ}{\Sigma} _{\overrightarrow {e}[E']}(C,f)$ in case $E < e <-E'$.

So now assume $E=e$ (hence $t>0$) and define
\begin{multline*}
\overset {\circ}{\Sigma}_{\overrightarrow {e}[E'],t}(C,f) = \{ L \in \overset {\circ}{\Sigma}_{\overrightarrow {e}[E']}(C,f) : \\
\rk (u(L)) \leq \deg (D)-\Sigma _{j\geq n} ((e_j-E')+E'-E+1)-t+d-E'k+(E'-E)k-g+1 \}
\end{multline*}
As a scheme structure on this set we take the intersection of $\overset{\circ}{\Sigma}_{\overrightarrow {e}[E]}(C,f)$ and the scheme defined by $\Fitt(\overrightarrow {e},-E)$.
It is the $t$-th degeneracy scheme of $\phi$.
Clearly $\overset {\circ}{\Sigma} _{\overrightarrow {e}[E]}(C,f) \subset \overset {\circ}{\Sigma} _{\overrightarrow {e}[E'],t}(C,f)$.
For $L' \in \overset {\circ}{\Sigma} _{\overrightarrow {e}[E'],t}(C,f)$ we have $L' \notin \overset {\circ}{\Sigma}_{\overrightarrow {e}[E]}(C,f)$ if and only if $h^0(C,L'-eM) > r(-e)+1$ for some $E \leq e <E'$, showing the inclusion is open as an inclusion of sets.
Using the expected codimensions of degeneracy loci for the homomorphism $u$ we have
\begin{multline*}
\codim _{\overset {\circ}{\Sigma}_{\overrightarrow {e}[E']}(C,f)}(\overset {\circ}{\Sigma}_{\overrightarrow {e}[E'],t}(C,f)) \\
\leq t \cdot (\Sigma _{i \geq n}((e_i-E')+E'-E+1)+t-d+kE'-(E'-E)k+g-1) \text { .} 
\end{multline*}
From some computations done in Lemma \ref{lemma2} it follows $\codim _{\overline{\Sigma}_{\overrightarrow {e}[E']}(C,f)}(\overline {\Sigma}_{\overrightarrow {e}[E]}(C,f))$ is equal to this upperbound.
Since $\Sigma _{\overrightarrow {e}[E]}(C,f)$ is an open dense subset of $\overline {\Sigma}_{\overrightarrow {e}[E]}(C,f)$ (Proposition \ref{prop1.1} (3)) we obtain at each point $L \in \overset{\circ}{\Sigma}_{\overrightarrow {e}[E]}(C,f)$ the degeneracy scheme $\overset {\circ}{\Sigma}_{\overrightarrow {e}[E'],t}(C,f)$ has the expected codimension at $L$.



Also by definition $L \in \overset{\circ}{\Sigma}_{\overrightarrow {e}[E]}(C,f)$ does not belong to the $(t+1)$-th degeneracy locus of $\phi$.
This  implies $\overset {\circ}{\Sigma}_{\overrightarrow {e}[E'],t}(C,f)$ is locally at $L$ a complete intersection in $\overset {\circ}{\Sigma} _{\overrightarrow {e}[E']}(C,f)$ and by the induction hypothesis it is locally a complete intersection.

\underline{Finishing the proof.}

The scheme $\overset {\circ}{\Sigma}_{\overrightarrow {e}[E'],t}(C,f)$ is the scheme $U$ mentioned at the beginning of the proof.
Since $\overset {\circ}{\Sigma}_{\overrightarrow {e}[E]}(C,f)\subset \overset {\circ}{\Sigma}_{\overrightarrow {e}[E'],t}(C,f)$ is open as a set we have that locally at $L \in \overset {\circ}{\Sigma}_{\overrightarrow {e}[E]}(C,f)$ the supports of $U$ and $\overline{\Sigma}_{\overrightarrow{e}[E]} (C,f)$ are equal (notice that also $\overset {\circ}{\Sigma}_{\overrightarrow {e}[E'],t}(C,f)$ is an open subset of $\overline{\Sigma}_{\overrightarrow{e}[E]} (C,f)$).
We also already noticed that $U$ is locally a complete intersection at  $L \in \overset {\circ}{\Sigma}_{\overrightarrow {e}[E]}(C,f)$.

From Proposition 4 in \cite {ref2} we know that the tangent space to $\overline {\Sigma}_{\overrightarrow {e}[E]}(C,f)$ at $L \in \Sigma _{\overrightarrow {e}[E]}(C,f)$ is the intersection of the tangent spaces to the schemes defined by the Fitting ideals $\Fitt (\overrightarrow {e},-e)$ with $e=e_i$ for some $e_i \geq E$.
But locally at $L$ those Fitting ideals define $U$.
Moreover from Proposition \ref{prop1.1} (5) we know $\Sigma _{\overrightarrow {e}[E]}(C,f)$ is smooth.
It follows $U$ is smooth at $L$ because $\dim T_L(U)=\dim_L (U)$, hence locally at $L \in \Sigma _{\overrightarrow {e}[E]}(C,f)$ the schemes $\overset {\circ}{\Sigma} _{\overrightarrow {e}[E]}(C,f)$ and $U$ are equal.
This shows the three properties mentioned at the beginning of the proof are fulfilled, hence the proposition is proved.
\end{proof}

We are going to give a stepwise argument to prove the claim on $\codim _{\overline {\Sigma}_{\overrightarrow {e}[E']}}(\overline{\Sigma}_{\overrightarrow {e}[E'],t})$ in the proof of Proposition \ref{prop1}.
For each integer $e$ we have $\codim _{\overline {\Sigma}_{\overrightarrow {e}[e-1]}} (\overline {\Sigma} _{\overrightarrow {e} [e]})=u(\overrightarrow {e}[e])-u(\overrightarrow {e}[e-1])$.
Let $t$ be the number of indices $i$ with $e_i=e-1$.
This integer is used in the proof of Proposition \ref{prop1} in case $E+1 \leq e \leq E'$.
The claim used in the proof of Proposition \ref{prop1} follows from the  following lemma.

\begin{lemma}\label{lemma2}
\[
u(\overline {e}[e])-u(\overline {e}[e-1])=t \cdot (\Sigma _{j : e_j \geq e}(e_j-e+2)+t-d+k(e-1)+g-1) \text { .}
\]
\end{lemma}

\begin{proof}
We write
\[
\overrightarrow {e}[e]=(-b-1, \cdots , -b-1, \overbrace{-b, \cdots ,-b}^y,e_n, \cdots , e_k)
\]
\[
\overrightarrow {e}[e-1]=(-b'-1, \cdots , -b'-1, \overbrace {-b', \cdots , -b'}^{y' }, \overbrace {e-1, \cdots , e-1}^t, e_n , \cdots , e_k)
\]
with the integer $n$ defined by the conditions $e_n \geq e$ and $e_{n-1} <e$.
We have the equalities
\begin{multline*}
g+k-d-1=by+(b+1)(n-1-y)-\Sigma_{i \geq n}e_i \\
= b(n-1)+n-1-y-\Sigma_{i \geq n}e_i
\end{multline*}
\begin{multline*}
g+k-d-1=b'y'+(b'+1)(n-1-t-y')-t(e-1)-\Sigma _{i \geq n}e_i\\
=b'(n-1-t)+n-1-y'-t-t(e-1)-\Sigma _{i \geq n}e_i \text { .}
\end{multline*}
This implies the equality
\begin{equation}\label{vgl1}
b'(n-1-t)=b(n-1)-y+y'+te \text { .}
\end{equation}

We have
\begin{multline*}
u(\overrightarrow {e}[e])=\Sigma _{i \geq n} [(e_i+b)(n-1-y)+(e_i+b-1)y]+C(e_i)\\
=\Sigma _{i \geq n} [(e_i+b)(n-1)]-y(k-n+1)+C(e_i)
\end{multline*}
with $C(e_i)=\Sigma _{n \leq j \leq j'} \max \{ e_{j'}-e_j-1, 0 \}$ and
\begin{multline*}
u(\overrightarrow{e}[e-1])=\Sigma _{i \geq n}[(e_i+b')(n-1-y'-t)+(e_i+b'-1)y']\\
+t[(e-1+b')(n-1-y'-t)+(e-2+b')y'+\Sigma _{i \geq n}(e_i-e)]+C(e_i)\\
=\Sigma _{i \geq n} [e_i(n-1-t)]+(k-n+1)b'(n-1-t)-y'(k-n+1)\\
+t[(e-1)(n-1-t)+b'(n-1-t)-y'+\Sigma _{i \geq n} (e_i-e)]+C(e_i) \text {  .}
\end{multline*}
Using Equation \ref{vgl1} twice in those expressions we find
\begin{multline*}
u(\overrightarrow {e}[e-1])-u(\overrightarrow{e}[e])=\Sigma_{i \geq n}[e_i(n-1-t)]+(k-n+1)(b(n-1)-y+y'+te)\\
-y'(k-n+1)+t[(e-1)(n-1-t)+b(n-1)-y+te+\Sigma _{i \geq n} (e_i-e)]\\
-\Sigma_{i \geq n}[(e_i+b)(n-1)]+y(k-n+1)=t[(e-1)(n-1-t)+b(n-1)-y+te] \text { .}
\end{multline*}
So we need to obtain
\[
(e-1)(n-1-t)+b(n-1)-y+te=(\Sigma _{i \geq n}e_i)-(k-n+1)(e-2)+t-d+k(e-1)+g-1 \text { .}
\]

Using $\Sigma _{i\geq n} e_i=b(n-1)+n-1-y-g-k+d+1$ we find this equality holds.
\end{proof}

As an application of Proposition \ref{prop1} we obtain smoothness result generalizing Proposition 7 in \cite{ref2}.
First we state the following corollary from the arguments used in the proof of Proposition \ref{prop1}.

\begin{corollary}\label{cor2.1}
Let $\overrightarrow{e}$ be a splitting sequence with $u(\overrightarrow{e})\leq g$ and let $C$ be a general $k$-gonal curve of genus $g$.
Let $e_1+2 \leq e' < e \leq e_k$.
Let $L \in \overset{\circ}{\Sigma}_{\overrightarrow {e}[e']}(C,f)$, hence also $L \in \overset{\circ}{\Sigma}_{\overrightarrow {e}[e]}(C,f)$.
Then locally at $L$ as a point on $\overset{\circ}{\Sigma}_{\overrightarrow{e}[e]}(C,f)$ the scheme $\overline{\Sigma}_{\overrightarrow{e}[e']}(C,f)$ is defined by $u(\overrightarrow {e}[e'])-u(\overrightarrow {e}[e])$ equations on $\overline{\Sigma}_{\overrightarrow{e}[e]}(C,f)$.
\end{corollary}
\begin{proof}
Proposition \ref{prop1} implies that locally at $L$ the scheme $\overline {\Sigma}_{\overrightarrow {e}[e']}(C,f)$ is the intersection of the schemes $\overline {\Sigma}_{\overrightarrow {e}[e]}(C,f)$ and the schemes defined by the Fitting ideals $\Fitt (\overrightarrow {e}, -j)$ with $e' \leq j' < e$.
In the proof of Proposition \ref{prop1} (including of course Lemma \ref{lemma2}) we obtained that each of those Fitting ideals gives a contribution of $u(\overrightarrow{e}[j])-u(\overrightarrow{e}[j+1])$ new equations.
This proves the corollary.
\end{proof}

\begin{proposition}\label{prop2}
Let $f : C \rightarrow \mathbb{P}^1$ be a general morphism of degree $d$ and genus $g$ and let $\overrightarrow {e}$ be a splitting type for degree $d$ and genus $g$.
For all integers $e \geq e_1+2$ $\Sigma _{\overrightarrow {e}}(C,f)$ is containd in the smooth locus of the scheme $\overline {\Sigma}_{\overrightarrow {e}[e]}(C,f)$.
\end{proposition}
\begin{proof}
For $L\in \Sigma _{\overrightarrow {e}}(C,f)$ we obtain from Corollary \ref{cor2.1}  that locally at $L$ the scheme $\overline {\Sigma}_{\overrightarrow {e}}(C,f)$ is given by $u(\overrightarrow {e})-u(\overrightarrow {e}[e])$ equations as a closed subscheme of $\overline {\Sigma} _{\overrightarrow {e}[e]}(C,f)$.
But because of Proposition \ref{prop1.1} (5) $\overline {\Sigma}_{\overrightarrow {e}}(C,f)$ is smooth at $L$.
Since $\codim _{\overline{\Sigma}_{\overrightarrow {e}[e]}(C,f)} \overline{\Sigma}_{\overrightarrow {e}}(C,f) = u(\overrightarrow {e})-u(\overrightarrow {e}[e])$ (see Proposition \ref{prop1.1}(1)) it follows $\overline{\Sigma}_{\overrightarrow {e}[e]}(C,f)$ is smooth at $L$.
\end{proof}

\begin{corollary}\label{cor3.1}
Let $f:C \rightarrow \mathbb{P}^1$ be a general covering of degree $k$ of a curve of genus $g$.
For all integers $e \geq e_1+2$ one has $\overset{\circ}{\Sigma}_{\overrightarrow {e}[e]}(C,f)$ is contained in the smooth locus of the scheme $\overline{ \Sigma}_{\overrightarrow {e}[e]}(C,f)$.
\end{corollary}

\begin{proof}
Let $\overrightarrow {e'} \leq \overrightarrow {e}$ with $e'_j =e_j$ for $j \geq i(\overrightarrow {e},e)$ and $e'_{i(\overrightarrow {e},e)-1} <e$.
We need to prove $\Sigma _{\overrightarrow {e'}}(C,f)$ is contained in the smooth locus of $\overline{\Sigma}_{\overrightarrow {e}[e]}(C,f)$.
Since $\overrightarrow {e'}[e]=\overrightarrow {e}[e]$ this follows from Proposition \ref{prop2}.
\end{proof}

\begin{example}\label{ex3.3}[Continuation of Examples \ref{ex3.1} and \ref{ex3.2}]

Let $g=20$ and $\overrightarrow {e}=(-4,-3,-1,-1,1,3)$.
It is noticed that $\overrightarrow {e}=\overrightarrow {e}[-1]$ hence $\overline{\Sigma}_{\overrightarrow {e}}(C,f)$ is not only smooth along $\Sigma _{e}(C,f)$ but also along $\Sigma_{\overrightarrow {e'}}(C,f)$ with $\overrightarrow {e'}=(-5,-2,-1,-1,1,3)$.
Taking $e=1$ we have $\overrightarrow {e}[1]=(-3,-2,-2,-2,1,3)$ and we find that $\overline{\Sigma}_{\overrightarrow {e}[1]}(C,f)$ is smooth along e.g. $\Sigma_{\overline{e'_1}}(C,f)$ and $\Sigma_{\overline{e'_2}}(C,f)$ but we don't know whether or not it is smooth along $\Sigma_{\overrightarrow{e'_5}}(C,f)$.
\end{example}

\begin{remark}\label{final}
Taking $e=0$ we find $\overline{\Sigma}_{\overrightarrow {e}[0]}(C,f)$ is smooth along $\overset{\circ}{\Sigma}_{\overrightarrow {e}[0]}(C,f)$.
This has the following interpretation.

The splitting sequence $\overrightarrow{e}$ implies lower bounds on $h^0(L+nM)$ for all $n \in \mathbb{Z}$ being sufficient and necessary conditions for $L \in \Pic (C)$ to belong to $\overline{\Sigma}_{\overrightarrow {e}}(C,f)$.
The smooth open subscheme $\Sigma_{\overrightarrow {e}}(C,f)$ contains exactly those invertible sheaves $L$ such that equality holds for all $n \in \mathbb{Z}$.

The set $\overline{\Sigma}_{\overrightarrow {e}}(C,f)$ is contained in $\overline{\Sigma}_{\overrightarrow {e}[0]}(C,f)$ and this is the set of invertible sheaves $L$ satisfying those lower bonds for all integers $n \leq 0$.
We find that the scheme $\overline{\Sigma}_{\overrightarrow {e}[0]}(C,f)$ is smooth at $L$ in case one has equality for all integers $n \leq 0$.

\end{remark}

As a special case we obtain the proof of Theorem C.
\begin{theorem}\label{TheoC}
Let $C$ be a general $k$-gonal curve of genus $g$ and assume $r$, $d$ and $l$ are such that $V^r_{d,l}(C,f)$ (see Definition \ref{defV}) is not empty.
The scheme $V^r_{d,l}(C,f)$ is smooth.
\end{theorem}

\begin{proof}
As mentioned in Remark \ref{remV} $V^r_{d,l}(C,f)$ is the scheme $\overset{\circ}{\Sigma}_{\overrightarrow {e}[0]}(C,f)$ with $\overrightarrow {e}$ of some type $B(e,b,y,m_1,m_2)$.
It follows from Corollary \ref{cor3.1} that $\overset {\circ}{\Sigma}_{\overrightarrow {e}[0]}(C,f)$ is contained in the smooth locus of $\overline{\Sigma}_{\overrightarrow{e}[0]}(C,f)$.
(As mentioned in Remark \ref{remV} we can restrict to cases with $e_1 \leq -2$).)
But in this case $\overrightarrow {e}[0]=\overrightarrow {e}$ and the scheme structure on $V^r_{d,l}(C,f)$ is the one induced by $\overline{\Sigma}_{\overrightarrow {e}}(C,f)$ so we conclude $V^r_{d,l}(C,f)$ is smooth as a scheme.
\end{proof}

\section{The smooth locus of the schemes $W^r_{d,f}(C)$.}\label{section3}

\begin{proposition}\label{prop3}
Let $f :C \rightarrow \mathbb{P}^1$ be a general morphism of degree $d$ with $C$ a smooth curve of genus $g$.
Then as schemes we have $\Sing (W^r_{d,f}(C))=W^{r+1}_{d,f}(C)$.
In particular, for every $L \in \Pic^d(C)_f$ the Petri map of $L$ is injective.
\end{proposition}
\begin{proof}
Let $L \in W^r_{d,f}(C)$.
The splitting sequence $\overrightarrow {e}$ of $L$ satisfies $e_k = \cdots = e_{k-r} = 0$.
This implies $\overrightarrow {e} \leq \overrightarrow {e'}$ with $\overrightarrow {e'}$ of type $B(0,b,y,r+1,0)$.
It follows from \cite{ref7} that  $\overline {\Sigma}_{\overrightarrow {e'}}(C,f)$  is an irreducible component of $W^r_d(C)$ of dimension $\rho^r_d(g)$ (in case $\rho^r_d(g)=0$ then is a finite set of isolated points of $W^r_d(C)$).
We obtain the equalities of sets
\[
W^r_{d,f}(C) = \overline {\Sigma}_{\overrightarrow {e'}}(C,f) \cap \Pic^d(C)_f
\]
and
\[
\overset {\circ}{\Sigma}_{\overrightarrow {e'}}(C,f)=(\overline {\Sigma}_{\overrightarrow {e'}}(C,f) \cap \Pic^d(C)_f) \setminus W^{r+1}_d(C) \text { .}
\]

From Proposition \ref{prop1.1} (4) we know the scheme $\overline {\Sigma}_{\overrightarrow {e'}}(C,f)$ is reduced.
Since $W^r_{d,f}(C)$ is equidimensional of dimension $\rho^r_d(g)$ it follows as a scheme $W^r_{d,f}(C)$ is Cohen-Macauly (notice that $W^r_{d,f}(C)$ is an open subset of $W^r_d(C)$).
In \cite {ref2} it is proved that the scheme $W^r_{d,f}(C)$ is generically reduced.
This implies because of the unmixedness Theorem (\cite {ref5}, Corollary 18.14) that also $W^r_{d,f}(C)$ is reduced as a scheme.
This implies the equality as schemes
\[
W^r_{d,f}(C) = \overline {\Sigma}_{\overrightarrow {e'}}(C,f) \cap \Pic^d(C)_f \text { .}
\]

In case $L \notin W^{r+1}_d(C)$ we have $\overrightarrow {e'}=\overrightarrow {e}[0]$.
It follows from Proposition \ref{prop2} that $L$ is contained in the smooth locus of $\overline{\Sigma}_{\overrightarrow {e'}}(C,f)$.
On the other hand, from \cite{ref9}, Chapter IV, Proposition 4.2 (ii), we know $W^{r+1}_{d,f}(C)$ is contained in the singular locus of $W^r_{d,f}(C)$ and we conclude that $\Sing (W^r_{d,f}(C))=W^{r+1}_{d,f}(C)$.

Now for any $L \in \Pic^d(C)_f$ we can take $r$ such that $L \in W^r_{d,f}(C) \setminus W^{r+1}_{d,f}(C)$.
So $L$ is a smooth point of $W^r_d(C)$ and locally at $L$ we have $\dim _L W^r_d(C)=\rho ^r_d(g)$.
It is proved in \cite {ref9} Chapter IV, Proposition 4.1(i), that this is equivalent to the Petri map at $L$ being injecctive.
\end{proof}

As a corollary we obtain Theorem A from the introduction.

\begin{theorem}\label{TheoA}
Let $(C,f)$ be a general $k$-gonal curve and let $L\in \Pic^d(C)$.
The Petri map of $L$ is injective if and only if $h^0(C,L-M)\cdot h^0(C,\omega _C-L-M)=0$.
\end{theorem}
\begin{proof}
Of course, $h^0(C,L-M)=0$ is equivalent to $L \in \Pic^d(C)_f$, therefore in this case it is proved in Proposition \ref{prop3} that the Petri map is injective.
Since the Petri maps of $L$ and $\omega_C-L$ are the same, it implies the Petri map of $L$ is also injective in case $h^0(C,\omega_C-L-M)=0$.

In case $h^0(C,L-M)\cdot h^0(C,\omega _C-L-M) \neq 0$, choose $u \in H^0(C,L-M)$ and $v \in H^0(C,\omega _C-L-M)$, both non-zero, and let $s, t$ be a base of $H^0(C,M)$.
Then clearly $su\otimes tv-tu \otimes sv$ belongs to the kernel of the Petri map of $L$.
\end{proof}

\begin{remark}\label{rem1}
The injectivity of the Petri map implies many local properties of the schemes $W^r_d(C)$ at $L \in W^r_{d,f}(C)$ in case $C$ is a general $k$-gonal curve in case $F$ has characteristic 0.
In particular, taking into account \cite{ref9}, Chapter IV, Proposition 4.1(iii), it implies the conditions of \cite{ref9}, Chapter VI, Theorem 2.1 are fulfilled, giving an explicit description of the tangent cone $TC_L(W^r_d(C))$ of $W^r_d(C)$ at $L$ and an explicit formula for its multiplicity at $0$.
But we can also use the results from \cite{ref16} giving much stronger information on the local structure of $W^r_d(C)$ at such point $L$ in case $C$ is a general $k$-gonal curve.
As a matter of fact, from \cite{ref16}, Theorem 1.4 it follows that locally at such point $L$ for the étale topology $W^r_d(C)$ is isomorphic to the tangent cone $TC_L(W^r_d(C))$ at $0$.

In particular the local results from Theorem 1.6 in \cite{ref16} can be applied in this case.
One of those results imply that for a general $k$-gonal curve $C$ the scheme $W^r_d(C)$ has a rational singularity at $L \in W^{r+1}_{d,f}(C)$.
It should be noted that in \cite{ref3} it is proved that all components of $W^r_d(C)$ have at most rational singularities in case $C$ is a general $k$-gonal curve.
Another example of an application is a formula for the log canonical threshold of $(\Pic ^d(C), W^r_d(C))$ at $L \in W^r_{d,f}(C)$.
\end{remark}

As already noticed, $W^r_{d,f}(C)$ is open in $W^r_d(C)$.
In particular, the closure of $W^r_{d,f}(C)$ is the union of irreducible components of $W^r_d(C)$ of type I (see Notation \ref{notfree}).

\begin{corollary}\label{cor2}
Let $C$ be a general $k$-gonal curve of genus $g$ and let $Z$ be an irreducible component of $W^r_d(C)$ of type 1.
Then 
\[
Z \cap \Sm (W^r_d(C))=((Z \cap \Pic^d(C)_f) \cup (Z \cap (\omega _C-\Pic^{2g-2-d}(C)_f)))\setminus W^{r+1}_d(C) \text { .}
\]
\end{corollary}
\begin{proof}
Let $L \in Z \cap \Sm (W^r_d(C))$.
From \cite {ref9} we know $W^{r+1}_d(C) \subset \Sing (W^r_d(C))$, hence $L \notin W^{r+1}_d(C)$.
Since $\dim (Z)=\rho ^r_d(g)$ it follows from \cite {ref9} that $L \in Z \setminus W^{r+1}_d(C)$ belongs to $\Sm (W^r_d(C))$ if and only if the Petri map of $L$ is injective.
From Theorem \ref{TheoA} we know this is equivalent to $L \in \Pic^d(C)_f$ or $L=\omega _C \otimes L'^{-1}$ with $L' \in \Pic^{2g-2-d}(C)_f$.
This implies the claim.
\end{proof}

Let $L \in \omega _C -\Pic^{2g-2-d}(C)_f$ or, equivalently, let $h^0(C,\omega _C-L-M)=0$.
This is equivalent to $h^0(C,\omega _C-L-M-P)=h^0(C,\omega _C-L-M)$ for each point $P \in C$.
Because for each point $P \in C$ we have $h^0(C,M-P) \neq 0$ it is equivalent to $h^0(C,\omega _C-L-M)=h^0(C,\omega _C-L-2M)$, hence $h^0(C,L+2M)=h^0(C,L+M)+k$.
Letting $\overrightarrow {e}$ be the splitting sequence of $L$ it is equivalent to $e_1 \geq -2$.
Therefore we have the following.

\begin{corollary}\label{cor3}
Let $C$ be a general $k$-gonal curve of genus $g$ and let $Z$ be an irreducible component of type I of $W^r_d(C)$.
Let $\overrightarrow {e}$ be the splitting sequence of a general element $L \in Z$ and assume $e_1<-2$.
We have
\[
Z \cap Sm(W^r_d(C))=(Z \cap \Pic^d(C)_f) \setminus W^{r+1}_d(C) \text { .}
\]
\end{corollary}

In case $e_1 <-2$ in Corollary \ref{cor3} we have the following more geometric explanation for points $L$ belonging to $(Z \cap \Sing (W^r_d(C))) \setminus W^{r+1}_d(C)$.

\begin{proposition}\label{prop4}
Let $C$ be a general $k$-gonal curve of genus $g$ and let $Z$ be an irreducible component of type I of $W^r_d(C)$ such that for $L \in Z$ general we have $\omega _C -L \notin \Pic^{2g-2-d}(C)_f$.
There exists another irreducible component $Z' \neq Z$ of $W^r_d(C)$ containing $Z \setminus (\Pic^d(C)_f \cup W^{r+1}_d(C))$.
\end{proposition}
\begin{proof}
Let $L \in Z \setminus (\Pic^d(C)_f \cup W^{r+1}_d(C))$ and let $\overrightarrow {e}$ be the splitting sequence of $L$.
Let $\overrightarrow {e'}$ be the splitting sequence of a general element of $Z$ (hence $\overrightarrow {e'}$ is of type $B(0,b,y,r+1,0)$).
We have $\overrightarrow {e} \leq \overrightarrow {e'}$ and by assumption on $Z$ we have $e'_1<-2$.

Since $L \notin \Pic^d(C)_f$ we have $e_k \geq 1$ (while $e'_k=0$).
Since $L \notin W^{r+1}_d(C)$ we need $r\geq 1$.
For all $1 \leq n \leq k$ we have $\Sigma _{i=1}^n e_i \leq \Sigma _{i=1}^n e'_i$ while $\Sigma _{i=1}^k e_i=\Sigma _{i=1}^k e'_i=\Sigma _{i=1}^{k-r-1}e'_i$.
Suppose there exists $k-r \leq n <k$ such that $\Sigma _{i=1}^n e_i=\Sigma _{i=1}^n e'_i$.

In particular we have $\Sigma _{i=1}^n e_i=\Sigma _{i=1}^{k-r-1} e'_i=\Sigma _{i=1}^k e_i$.
It follows that $\Sigma _{i=k-r}^n e_i=\Sigma _{i=1}^ne_i -\Sigma _{i=1}^{k-r-1} e_i=\Sigma _{i=1}^{k-r-1} e'_i -\Sigma _{i=1}^{k-r-1}e_i \geq 0$.
This implies $e_i \geq 0$ for $n \leq i < k$ and since $e_k \geq 1$ we have $\Sigma _{i=k-r}^k e_i > \Sigma _{i=1}^n e_i -\Sigma _{i=1}^{k-r-1} e_i=\Sigma _{i=1}^k e_i -\Sigma _{i=1}^{k-r-1}e_i$, giving a contradiction.
Therefore for each $k-r \leq n <k$ we have $\Sigma _{i=1}^ne_i < \Sigma _{i=1}^n e'_i$.

Let $\overrightarrow {e''}$ be the splitting sequence defined by $e''_i=e'_i$ for $i \leq k-r-1$, $e''_{k-r}=-1$, $e''_i=0$ for $k-r<i<k$ and $e''_k=1$, then we find $\overrightarrow {e} \leq \overrightarrow {e''}$, hence $\Sigma_{\overrightarrow {e}}(C,f) \subset \overline {\Sigma}_{\overrightarrow {e''}}(C,f)$.
Since $e_1<-2$ it follows that $\overrightarrow {e''}$ is not of balanced plus balanced type.
Let $\overrightarrow {e'''}$ be the balanced plus balanced type obtained from $\overrightarrow {e''}$ taking $e'''_i=e''_i$ for $i \geq k-r+1$.
We have $\overrightarrow {e''} \leq \overrightarrow {e'''}$, hence $\overline {\Sigma} _{\overrightarrow{e''}}(C,f) \subset \overline {\Sigma}_{\overrightarrow {e'''}}(C,f)$ and $e'''_{k-r}<-1$.
This implies $\overline {\Sigma}_{\overrightarrow {e'''}}(C,f)$ is an irreducible component $Z' \neq Z$ of $W^r_d(C)$ and $L \in Z'$. 
\end{proof}

Assume $C$ is a $k$-gonal curve and a general element $L\in W^r_{d,f}(C)$ has a splitting sequence $\overrightarrow {e}$ with $e_1=-1$.
In that case one has $W^r_d(C)=\Pic^d(C)$ with $d=g+r$.
Therefore in considering the case in Corollary \ref{cor2} with $e_1 \geq -2$ we can restrict to $e_1=-2$.
In case $e_2=-1$ we are considering $W^r_d(C)=\omega _C-W^0_{g-r-1}(C)$ (hence $d=g+r-1$).
From our results it follows that then for a general $k$-gonal curve $C$ we have $\Sm (W^r_d(C))=W^r_d(C) \setminus W^{r+1}_d(C)$ (as soon as $L \in W^r_d(C) \setminus W^{r+1}_d(C)$ then $\omega _C -L$ is free from $M$).
So we can restrict to the case $e_1=e_2=-2$ and in this case we obtain the following result.

\begin{proposition}\label{prop5}
Let $C$ be a general $k$-gonal curve and let $Z$ be an irreducible component of type I such that a general element $L$ of $Z$ has a splitting sequence $\overrightarrow{e}$ with $e_1=e_2=-2$.
If the genus $g$ is large enough then there exists $L \in Z \setminus W^{r+1}_d(C)$ such that $W^r_d(C)$ is singular at $L$ and $Z$ is the only irreducible component of $W^r_d(C)$ containing $L$ (in particular $Z$ is singular at $L$ as a set).
\end{proposition}
\begin{proof}
Define the integer $l \geq 2$ such that $e_l=-2$ and $e_{l+1}>-2$.
Clearly $l \leq k-r-1$.
Consider the splitting sequence $\overrightarrow{e'}$ defined by $e'_1=-3$, $e_i=-2$ for $2 \leq i \leq l-1$, $e'_i=-1$ for $l\leq i \leq k-r$, $e'_i=0$ for $k-r+1 \leq i <k$ and $e'_k=1$.
Clearly $\overrightarrow {e'} \leq \overrightarrow {e}$, hence $\Sigma_{\overrightarrow {e'}}(C,f)\subset Z\setminus W^{r+1}_d(C)$ and if $g$ is large enough it is not empty.
Also for $L \in \Sigma _{\overrightarrow {e'}}(C,f)$ we have $h^0(C,L-M).h^0(C,\omega _C-L-M)\neq 0$, hence $L \in \Sing(W^r_d(C))$.

Assume such element $L$ would belong to another irreducible component $Z'$ of $W^r_d(C)$.
Let $L'$ be a general element of $Z'$ and let $\overrightarrow {e''}$ be its splitting sequence.
It has to be of balanced plus balanced type.
Since $Z' \neq Z$ we have $L' \notin \Pic^d(C)_f$, hence $h^0(C,L'-M)\neq 0$.
Since $h^0(C,L-M)=1$ we have $h^0(C,L'-M)=1$.
It implies $e''_k=1$ and $e''_{k-1}<1$.
Since $h^0(C,L')=r+1$ it implies $e''_i=0$ for $k-r+1 \leq i \leq k-1$ and $e''_{k-r}<0$.
This implies $\overrightarrow {e''}$ has to satisfy $e''_i=e_i$ in case $1 \leq i \leq k-r-1$ and $e''_{k-r}=-1$.
This implies $\overrightarrow {e''} \leq \overrightarrow {e}$, hence $Z' \subset Z$, a contradiction.
As a conclusion we find $Z$ is the only irreducible component of $W^r_d(C)$ containing $L$.
\end{proof}

We consider some examples illustrating the previous result and indicating how to make the expression ''$g$ large enough'' more concrete.

\begin{example}\label{ex1}[The singularity locus of the theta divisor]

We consider the scheme $W^1_{g-1}(C)$ for a general $k$-gonal curve $C$.

A general element $L$ in some component of $W^1_{g-1}(C)$ has a splitting sequence $\overrightarrow {e}$ with $\Sigma _{i=1}^k e_i=-k$.
In case $e_k=1$ it follows $\Sigma _{i=1}^{k-1} e_i=-(k+1)$.
In case $e_{k-1}<-1$ then $\Sigma _{i=1}^{k-1} e_i\leq -2(k-1)$ and we obtain $k\leq 3$.
This shows that in case $k\geq 4$ then $W^1_{g-1}(C)$ is irreducible of dimension $g-4$ and it is of type I.
In particular $\overrightarrow {e}= (-2,-2,-1, \cdots ,-1,0,0)$.

Let $\overrightarrow{e'}=(-3,-1,\cdots ,-1,1)$.
We have $\overline {\Sigma}_{\overrightarrow {e'}}(C,f)$ is not empty if and only if $g \geq 2k-1$.
It follows that in case $g<2k-1$ we have $\Sing (W^1_{g-1}(C))=W^2_{g-1}(C)$ and in case $g \geq 2k-1$ we have $\Sing (W^1_{g-1})\setminus W^2_{g-1}(C)$ is non-empty and if $g>2k-1$ it is irreducible of dimension $g-(2k-1)$ (it is equal to $\overline {\Sigma}_{\overrightarrow{e'}}(C,f) \setminus W^2_{g-1}(C)$).

Those singular points of $W^1_{g-1}(C)$ are obtained as follows.
One has $\dim (\vert \omega _C-2M \vert)=g-(2k-1)$.
A general element $E$ of $\vert \omega _C -2M \vert$ can be written as $E=E_1+E_2$ with $E_1$, $E_2$ both effective of degree $g-1-k$, $\dim \vert M+E_1 \vert =1$ and $L=M+E_1$ is a general element of $\Sigma _{\overrightarrow {e'}}(C,f)$.
This is proved by \cite{ref7} for this splitting sequence $\overrightarrow {e'}$.
\end{example}

\begin{example}\label{ex2}
One can make a similar description in case $\overrightarrow {e}=(-2,-2,-2,-1,\cdots ,-1,0,0)$ for $k \geq 5$.
In case $g \geq 7$ we have $\overline {\Sigma}_{\overrightarrow {e}}(C,f)$ is an irreducible component $Z$ of $W^1_{g-2}(C)$ of type I of dimension $g-6$ in case $C$ is a general $k$-gonal curve.
Let $\overrightarrow {e'}=(-3,-2,-1, \cdots, -1,1)$.
In case $g \geq 2k-1$ we have $\overline {\Sigma}_{\overrightarrow {e'}}(C,f)$ is not empty of dimension $g-(2k-1)$ and $\overline {\Sigma}_{\overrightarrow {e'}}(C,f) \setminus W^2_{g-2}(C)=\Sing (Z)\setminus W^2_{g-2}(C,f)$.
Again a general element $E$ of $\vert \omega _C-2M \vert$ can be written as $E=E_1+E_2$ with $E_1$ effective of degree $g-2-k$ and $E_2$ effective of degree $k$ with $\dim (\vert M+E_1 \vert )=1$ and $L=M+E_1$ is a general point of $\Sigma _{e'}(C,f)$.
\end{example}

\begin{example}\label{ex3}
The description of this example is a little bit more complicated.
Let $k=7$ and consider $\overrightarrow {e}=(-2,-2,-2,-2,0,0,0)$.
In case $g \geq 12$ then $\overline{\Sigma}_{\overrightarrow {e}}(C,f)$ is an irreducible component $Z$ of type I of dimension $g-12$ of $W^2_{g-2}(C)$ in case $C$ is a general 7-gonal curve of genus $g$.
Let $\overrightarrow {e'}=(-3,-2,-2,-1,-1,0,1)$.
In case $g \geq 15$ then $\Sigma_{\overrightarrow {e'}}(C,f)$ is contained in the singular locus of $Z \setminus W^2_{g-2}(C)$.

An element of $\Sigma _{\overrightarrow {e'}}(C,f)$ is of the type $M+E_1$ with $E_1$ effective of degree $g-9$, $\vert E_1-M \vert = \emptyset$ and $\dim (\vert M+E_1 \vert )=2$.
We also need to have $\dim (\vert 2M+E_1 \vert)=6$, hence $\dim (\vert \omega _C -2M-E_1 \vert )=0$.

We have $\dim (\vert \omega _C -2M \vert)=g-13$ and inside $\vert \omega _C-2M \vert$ we have a subspace of dimension $g-15$ consisting of divisors $E$ that can be written as $E=E_1+E_2$ with $E_1$ effective of degree $g-9$, $E_2$ effective of degree $g-7$ and having the special property $\dim (\vert M+E_1 \vert )=2$.
One has $L=M+E_1$ is a general element of $\Sigma_{\overrightarrow {e'}}(C,f)$.
\end{example}

\begin{remark}\label{rem3}
Using the notation from the proof of Proposition \ref{prop4}, by definition we have $\overrightarrow {e''} \leq \overrightarrow {e'}$ hence $\Sigma _{\overrightarrow {e''}}(C,f) \subset Z \setminus W_d^{r+1}(C)$.
We showed $W^r_d(C)$ is singular along $\Sigma_{\overrightarrow {e''}}(C,f)$ because it is in the other irreducible component $Z'$ of $W^r_d(C)$.
But restricting to the component $Z$ of $W^r_d(C)$ we have $\Sigma_{\overrightarrow {e''}}(C,f)$ belongs to the smooth locus of $Z$.
Indeed, $\Sigma_{\overrightarrow {e''}}(C,f) \subset \overline {\Sigma}_{\overrightarrow {e'}}(C,f)$ has codimension 1 and in \cite{ref6}, Theorem 9.1 it is proved that this implies $\Sigma _{\overrightarrow {e''}}(C,f) \subset \Sm (\overline{\Sigma}_{\overrightarrow {e'}}(C,f))$.
In the next section we give a generalisation of this statement based on the use of Petri maps.
\end{remark}

\section{A smoothness result for the degeneracy loci}\label{section4}

In order to prove the irreducibility statement contained in Proposition \ref{prop1.1} (4), the authors prove the scheme $\overline {\Sigma} _{\overrightarrow {e}}(C,f)$ is smooth along $\Sigma _{\overrightarrow {e'}}(C,f)$ in case $\overrightarrow {e'} \leq \overrightarrow {e}$ and $u(\overrightarrow {e'})=u(\overrightarrow {e})+1$ (i.e. the codimension 1 case).
As a first step of the proof in \cite{ref6} in proving that, the authors describe such pairs $\overrightarrow {e'} \leq \overrightarrow {e}$.
Those pairs are described by the existence of $1 \leq i < i-1+k' \leq k$ with $k' \geq 2$ such that $e_i = \cdots =e_{i-1+k'}=a$, $e_{i-1} <a-1$ in case $i>1$, $e_{i+k'} > a+1$ in case $i+k'\leq k$ and $e'_i=a-1$, $e_{i-1+k'}=a+1$ and $e'_j=e_j$ in case $j \notin \{ i, i-1+k' \}$.

As a generalisation we prove the following
\begin{proposition}\label{prop5.1}
Let $\overrightarrow {e}$ be such that there exists $1 \leq i \leq i-1+k' \leq k$ such that $e_i = \cdots = e_{i-1+k'}=a$, $e_{i-1}<a-1$ in case $i >1$ and $e_{i+k'}>a+1$ in case $i+k' \leq k$.
Let $l$ be an integer satisfying $1 \leq l \leq [\frac{k'}{2}]$ and consider $\overrightarrow {e'}$ defined by $e'_j=a-1$ for $i \leq j \leq i+l-1$, $e'_j=a+1$ for $i+k'-l \leq j \leq i-1+k'$ and $e'_j=e_j$ in all other cases.
In case $C$ is a general $k$-gonal curve of genus $g$ and $u(\overrightarrow {e'})\leq g$ then $\Sigma _{\overrightarrow {e'}}(C,f)$ is contained in the smooth locus of $\overline{\Sigma}_{\overrightarrow {e}}(C,f)$.
\end{proposition}

\begin{proof}
One has $u(\overrightarrow{e'})-u(\overrightarrow {e})=l^2$ hence $\dim (\overline{\Sigma}_{\overrightarrow {e}}(C,f))- \dim (\Sigma _{\overrightarrow {e'}}(C,f))=l^2$.
Also $\Fitt(\overrightarrow {e'},n)=\Fitt (\overrightarrow {e},n)$ in case $n \neq -(a+1)$.
Taking $L \in \Sigma _{\overrightarrow {e'}}(C,f)$, it follows that
\[
T_L (\overline {\Sigma}_{\overrightarrow {e'}}(C,f))=T_L (\overline {\Sigma}_{\overrightarrow {e}}(C,f))\cap T_L(Z(\Fitt (\overrightarrow {e},-(a+1)))) \text { .}
\]

Consider the Petri map
\[
\mu _0 : H^0(C,L-(a+1)M) \otimes H^0(C,K-L+(a+1)M) \rightarrow H^0(C,K) 
\]
hence $T_L(Z(\Fitt (\overrightarrow {e},-(a+1))))=(\im \mu_0)^{\perp} \subset H^1(C, \mathcal{O}_C)$.
(We use that the identification of $T_L(\Pic^d(C))$ and $T_{L-(a+1)M}(\Pic ^{d-(a+1)k}(C))$ with $H^1(C,\mathcal{O}_C)$ are compatible with the tangent map of the translation $\Pic^d(C) \rightarrow \Pic^{d-(a+1)k}(C) : L \rightarrow L-(a+1)M$ (see \cite{ref2}, Remark 6).)
Consider the Petri map
\[
\mu _1 : H^0(C,L-aM) \otimes H^0(C,K-L+aM) \rightarrow H^0(C,K) \text { .}
\]
Let $s, t$ be a base for $H^0(C,M)$ and let $V= s \cdot H^0(C,K-L+aM)+t \cdot H^0(C,K-L+aM) \subset H^0(C,K-L+(a+1)M))$.
For $u \in H^0(C,L-(a+1)M)$, $x,y \in H^0(C,K-L+aM)$ one has 
\[
\mu _0 (u \otimes (sx+ty))= \mu _1 (su \otimes x + tu \otimes y)
\]
proving $\mu_0 (H^0(C,L-(a+1)M)\otimes V) \subset \im (\mu _1)$.
In case $z \in s\cdot H^0 (C,K-L+aM) \cap t \cdot H^0(C,K-L+aM)$ then there exist $x,y \in H^0(C,K-L+aM)$ with $z =sx = ty$.
Since $Z(s) \cap Z(t) = \emptyset$ there exists $z' \in H^0(C,K-L+(a-1)M)$ with $z=stz'$, hence 
\[
s \cdot H^0(C, K-L+aM) \cap t \cdot H^0(C,K-L+aM) = st \cdot H^0(C,K-L+(a-1)M) \text { .}
\]
This proves $\dim V=2h^0(K-L+aM)-h^0(K-L+(a-1)M)$.

Notice that $h^0(L-(a-1)M)-h^0(L-aM)=k-i+1$.
Therefore using Riemann-Roch we find $\dim V=h^0(L-aM)-(k-i+1)-d+g+(a+1)k-1$.
Also $h^0(L-aM)-h^0(L-(a+1)M)=k-i+1-l$ and again using Riemann-Roch we find
\[
\dim V=h^0(K-L+(a+1)M)-l \text { .}
\]

Next we consider the Petri map
\[
\mu_2 : H^0(C,L-(a+2)M) \otimes H^0(C, K-L+(a+2)M) \rightarrow H^0(C,K) \text { .}
\]
Let $V'=s \cdot H^0(C,L-(a+2)M) + t \cdot H^0(C,L-(a+2)M) \subset H^0(C,L-(a+1)M)$.
As before one finds $\mu _0 (V' \otimes H^0(C,K-L+(a+1)M)) \subset \im \mu _2$ and also
\[
s.H^0(C,L-(a+2)M) \cap t \cdot H^0(C,L-(a+2)M) = st \cdot H^0(C,L-(a+3)M) \text { .}
\]
This implies $\dim V'=2h^0(L-(a+2)M)-h^0(L-(a+3)M)$.

Now we use $h^0(L-(a+2)M)=h^0(L-(a+3)M)+k-(i-1+k')$, hence $\dim V' = h^0(L-(a+2)M)+k-(i-1+k')$.
Since $h^0(L-(a+1)M)-h^0(L-(a+2)M)=k-(i-1+k')+l$ we find
\[
\dim V' = h^0(L-(a+1)M)-l \text { .}
\]

Now fix subspaces $W$ (resp. $W'$) of dimension $l$ of $h^0(C,K-l+(a+1)M)$ (resp. $H^0(C,L-(a+1)M)$) such that $H^0(C,L-(a+1)M)=V' \oplus W'$ and $H^0(C,K-L+(a+1)M)=V \oplus W$.
We obtain $\im \mu_0 \subset \im \mu_1 + \im \mu_2 + \mu_0 (W' \otimes W)$.
This implies
\[
\dim (T_L(\overline{\Sigma} _{ \overrightarrow{e'}}(C,f))) \geq \dim (T_L (\overline{\Sigma}_{\overrightarrow {e}}(C,f)))-l^2 \text { .}
\]
We know from \cite{ref7} that $\overline{\Sigma}_{\overrightarrow {e'}}(C,f)$ is smooth at $L$, hence we obtain
\[
\dim (\overline{\Sigma}_{\overrightarrow {e'}}(C,f))+l^2 \geq \dim (T_L (\overline{\Sigma}_{\overrightarrow {e}}(C,f))) \text { .}
\]
Since $\dim (\overline{\Sigma}_{\overrightarrow {e'}}(C,f)+l^2 =\dim (\overline{\Sigma}_{\overrightarrow {e}}(C,f)$ this implies $\dim (T_L (\overline{\Sigma} _{\overrightarrow {e}}(C,f))=\dim \overline{\Sigma}_{\overrightarrow {e}}(C,f)$, i.e. $\overline {\Sigma }_{\overrightarrow {e}}(C,f)$ is smooth at $L$.
\end{proof}

It should be noted that more smoothness results of this kind can be proved using Petri maps in a similar way.

\begin{proposition}\label{propE1}
Let $Z$ be a component of type II of $W^r_d(C)$ on a general $k$-gonal curve $C$ and let $Z=\overline{\Sigma}_{\overrightarrow {e}}(C,f)$ with $\overrightarrow {e}$ of type $B(0,b,y,u,v)$ with $u>0$, $v>0$.
Let $L \in \overset {\circ}{\Sigma}_{\overrightarrow {e}}(C,f^)$ such that $\overrightarrow {e}(L)$ satisfies $e_{k-u-v} \leq -2$.
Then $L$ is a smooth point of the scheme $W^r_d(C)$.
\end{proposition}

\begin{proof}
From Proposition \ref{prop1} we know locally at $L$ the scheme $\overline{\Sigma}_{\overrightarrow {e}}(C,f)$ is defined by $\Fitt (\overrightarrow {e},-1)+\Fitt (\overrightarrow {e},0)$.
Also from Corollary \ref{cor3.1} we know the scheme $\overline{\Sigma}_{\overrightarrow {e}}(C,f)$ is smooth at $L$.
The Zariski tangent space $T_L(Z(\Fitt(\overrightarrow {e},-1)))$ is equal to $(\im (\mu_1))^{\perp}$ with $\mu _1$ being the Petri map
\[
\mu_1 : H^0(C,L-M) \otimes H^0(C,K-L+M) \rightarrow H^0(C,K) \text { .}
\]
From Theorem A we know $\mu _1$ is injective.
Also the Zariski tangent space $T_L (Z(\Fitt (\overrightarrow {e},0)))$ is  equal to $(\im \mu_2)^{\perp}$ with $\mu_2$ being the Petri map
\[
\mu_2 : H^0(C,L) \otimes H^0(C,K-L) \rightarrow H^0(C,K) \text { .}
\]

Consider a base $\{ s,t \}$ for $H^0(C,M)$ and consider the subspace
\[
s \cdot H^0(C,K-L)+t \cdot H^0(C,K-L) \subset H^0(C,K-L+M) \text { .}
\]
Using arguments as before one finds
\[
s \cdot H^0(C,K-L)+t \cdot H^0(C,K-L)=st \cdot H^0(C,K-L-M)
\]
hence
\begin{multline*}
\dim (s \cdot H^0(C,K-L)+t \cdot H^0(C,K-L))=2h^0(K-L)-h^0(K-L-M) \\
=v-d+k+g-1=h^0(K-L+M)
\end{multline*}
(note that in the second equality we use $e_{k-u-v}\leq -2$).
So we obtain
\[
s \cdot H^0(C,K-L)+ t \cdot H^0(C,k-L) = H^0(C,K-L+M) \text { .}
\]

An element $z$ of $\im (\mu_1)$ is of type $z=\Sigma_{i=1}^s x_iy_i$ with $x_i \in H^0(C,L-M)$ and $y_i \in H^0(C,K-L+M)$ (for some integer $s \geq 1$).
But we found there exists $y'_i, y''_i \in H^0(C,K-L)$ such that $y_i=s\cdot y'_i+t\cdot y''_i$, hence $z = \Sigma _{i=1}^s ((sx_i)y'_i+(tx_i)y''_i) \in \im (\mu _2)$.
This shows $\im (\mu_1 ) \subseteqq \im (\mu_2)$ hence $T_L(\overline {\Sigma}_{\overrightarrow {e}}(C,f))=(\im (\mu _2))^{\perp}$.

Since $\overline{\Sigma}_{\overrightarrow {e}}(C,f)$ is smooth at $L$ it implies $\dim (\im (\mu_2))=g-\dim (\overline{\Sigma}_{\overrightarrow {e}}(C,f))$.
Since $\overline{\Sigma}_{\overrightarrow {e}}(C,f) \subset W^r_d(C)$ we have $\dim (\overline{\Sigma}_{\overrightarrow {e}}(C,f)) \leq \dim _L (W^r_d(C))$.
Since $(\im (\mu_2))^{\perp} = T_L(W^r_d(C))$ we obtain $g-\dim (T_L(W^r_d(C)) = g-\dim(\overline{\Sigma}_{\overrightarrow {e}}(C,f)) \geq g-\dim _L(W^r_d(C))$, hence $\dim (T_L (W^r_d(C))) \leq \dim _L (W^r_D(C))$.
This implies equality and therefore $W^r_d(C)$ is smooth at $L$.
\end{proof}

The following proposition enables us to use Proposition \ref{propE1} in order to obtain smoothness results for components of type III of $W^r_d(C)$ on a general $k$-gonal curve $C$.
We give a more general formulation than needed valid for all $k$-gonal curves.

\begin{proposition}\label{prop5.2}
Let $f :C \rightarrow \mathbb{P}^1$ be a covering of degree $k$ with $C$ a smooth curve of genus $g$.
Let $M$ be the invertible sheaf associated to $f$.
Let $L$ be a smooth point of the scheme $W^r_d(C)$ (we assume $r>d-g$).
Let $\overrightarrow {e}$ be the splitting locus of $L$ with respect to $f$.
Assume $e_i \geq -2$ implies $e_i \geq 0$ (i.e. -2 and -1 do not occur in $\overrightarrow {e}$.
Let $s+1 = h^0(L+M)$ then $W^s_{d+k}(C)$ is smooth at $L+M$.

Moreover in that case if $Z$ is the irreducible component of $W^r_d(C)$ containing $L$ then $Z+M$ is the irreducible component of $W^s_{d+k}(C)$ containing $L+M$.
\end{proposition}

\begin{proof}
Since $W^r_d(C)$ is smooth at $L$ it follows $h^0(L)=r+1$.
Let $a = \# \{ i : e_i \geq 0 \}$.
From the assumpions that -1 and -2 do not occur in $\overrightarrow {e}$ it follows that $h^0(L+M)=r+1+a$ (hence $s=r+a$) and $h^0(M+2M)=r+1+2a$.
Consider the Petri map
\[
\mu_1 : H^0(C,L) \otimes H^0(C,K-L) \rightarrow H^0(C,K)
\]
then $W^r_d(C)$ being smooth at $L$ is equivalent to $\dim (Z)=g- \dim (\im (\mu_1))$.

Consider the Petri map
\[
\mu_2 : H^0(C,L+M) \otimes H^0(C,K-L-M) \rightarrow H^0(C,K)\text { .}
\]
We know $\dim (T_L(W^s_{d+k}(C))=g-\dim (\im (\mu_2))$.
Since $Z+M \subset W^s_{d+k}(C)$ it is enough to prove $g-\dim (\im (\mu_2)) \leq \dim (Z)$.

As before, let $\{ s,t \}$ be a base for $H^0(C,M)$ and consider 
\[
V=s\cdot H^0(C,K-L-M)+t\cdot H^0(C,K-L-M) \subset H^0(C,K-L) \text { .}
\]
For $x \in H^0(C,L)$ and $v=sy_1+ty_2$ with $y_1, y_2 \in H^0(C,K-L-M)$ we have $\mu_1 (x \otimes v)=\mu_2 (xs \otimes y_1 + xt \otimes y_2)$, hence $\mu_1 (H^0(C,L) \otimes V) \subset \im (\mu_2)$.

Repeating arguments from before, we know $V=st\cdot H^0(C,K-L-2M)$ and therefore $\dim (V)=2h^0(K-L-M)-h^0(K-L-2M)$.
Using the Riemann-Roch Theorem this implies $\dim (V)=(r+1)-d+g-1=h^0(K-L)$, hence $V=H^0(C,K-L)$.
So we found $\im (\mu_1) \subset \im (\mu_2)$ and therefore $g-\dim(\im (\mu_2))\leq g-\dim (\im (\mu_1))=\dim (Z)$.
\end{proof}

\begin{remark}\label{rem5.1}
In case some $e_i=-1$ occurs in $\overrightarrow {e}$ in the situation of Proposition \ref{prop5.2}, then we have $h^0(L+M)>r+1+a$ and therefore $L+M$ is a singular point of $W^{r+a}_{d+k}(C)$.
Therefore we do not consider this case.
The next proposition considers the case where no $e_i$ is equal to -1 but some $e_i$ is equal to -2.
\end{remark}

\begin{proposition}\label{prop5.3}
Let $C$, $f$, $L$, $Z$ and $\overrightarrow {e}$ be as in Proposition \ref{prop5.2}.
Assume $e_i \geq -1$ implies $e_i \geq 0$ and there exists some $e_i=-2$.
Then at least one of the two possibilies hold
\begin{itemize}
\item $Z+M \subset W^s_{d+k}(C)$ is not an irreducible component of $W^s_{d+k}(C)$.
\item $L+M$ is contained in two irreducible components of $W^s_{d+k}(C)$.
\end{itemize}
\end{proposition}

\begin{proof}
Let $i_0 = \max \{i: e_i=-2 \}$ (hence $e_{i_0+1} \geq 0$) and $j_0= \min \{ i : e_i=e_k \}$.
Consider the splitting sequence $\overrightarrow {e'}$ defined by $e'_i=e_i$ in case $i \notin \{ i_0,j_0 \}$ while $e'_{i_0}=-1$ and $e'_{j_0} = e_k-1$.

Since $\Sigma _{\overrightarrow {e'}}(C,f) \not\subset W^r_d(C)$ we have $\overline{\Sigma}_{\overrightarrow {e'}+1}(C,f) \not\subset Z+M$.
On the other hand $L+M \in \overline{\Sigma}_{\overrightarrow {e'}+1}(C,f) \subset W^s_{d+k}(C)$ and $L+M \in Z+M$.
Assume $Z+M$ is an irreducible component of $W^s_{d+k}(C)$.
Then we obtain the existence of an irreducible component $Z'$ of $W^s_{d+k}(C)$ different from $Z+M$ and containing $\overline{\Sigma}_{\overrightarrow {e'}+1}(C,f)$.
We obtain in that case that $L+M$ is contained in two irreducible components of $W^s_{d+k}(C)$.
\end{proof}

Theorem B from the introduction is the equivalence of (1) and (2) in the next theorem.
\begin{theorem}\label{theorem5.1}
Let $C$ be a general $k$-gonal curve of genus $g$.
Let $V^r_{d,l}(C,f) \subset W^r_d(C)$ be such that its closure is an irreducible component $Z$ of $W^r_d(C)$ (see Remark \ref{remV} and Remark \ref{types}).
For $L \in V^r_{d,l}(C,f)$ the following statements are equivalent
\begin{enumerate}
\item As a scheme $W^r_d(C)$ is not smooth at $L$.
\item $W^r_d(C)$ contains an irreducible component $Z'$ different from $Z$ with $L \in Z'$.
\item $h^0(L+M)>2r+2-l$.
\end{enumerate}
\end{theorem}

\begin{proof}
We use the notations $e$, $m_1$ and $m_2$ as in Remark \ref{remV}.
By definition $L'=L-eM$ belongs to an irreducible component $Z'$ of type I or II of $W^{2m_1+m_2-1}_{d-ek}(C)$ and $Z=Z'+eM$ (see Remark \ref{types}).
We write $\overrightarrow {e'}$ to denote the splitting type of $L'$, then the splitting type of $L$ is $\overrightarrow {e}=\overrightarrow {e'}+e$.
From Theorem A and Proposition \ref{propE1} it follows the scheme $W^{2m_1+m_2-1}_{d-ek}(C)$ is smooth at $L'$ (the condition $e_{k-u-v} \leq -2$ holds in case $v\geq 1$ because we assume $Z$ is an irreducible component of $W^r_d(C)$).
From Proposition \ref{prop5.2} it follows $W^r_d(C)$ is smooth at $L$ in case $e'_i \geq -1-e$ implies $e'_i\geq 0$ (one needs to perform an induction argument).
From the assumptions we know $e_i \geq 0$ implies $e_i \geq e$, hence $e'_i \geq -e$ implies $e'_i \geq 0$.
Therefore in case $W^r_d(C)$ is not smooth at $L$ then there is some $e'_i=-1-e$ hence $e_i=-1$.
This is equivalent to condition (3) hence we proved (1) $\Rightarrow$ (3).

But in case $e_i=-1$ it follows from Proposition \ref{prop5.3} applied to $L+(e-1)M \in Z'+(e-1)M$ that either $Z'+eM=Z$ is not an irreducible component of $W^r_d(C)$ (by assumption this is not the case) or $L$ is contained in at least 2 irreducible components of $W^r_d(C)$.
This proves the implication (3) $\Rightarrow$ (2).

Of course the implication (2) $\Rightarrow$ (1) is trivial.
\end{proof}

Finally we give an example illustrating Theorem \ref{theorem5.1}.
\begin{example}\label{lastEx}
Assume $C$ is a general 6-gonal curve of genus 50 and consider the irreducible component $Z$ of $W^4_{45}(C)$ corresponding to the splitting sequence $\overline{e} = (-4,-3,-3,-3,1,2)$ (hence $\dim (Z)=20$).
Using the notations of Remark \ref{remV} we know $Z$ is the closure of $V^4_{45,3}(C,f)$ (as usual, we write $f$ to denote a covering $f : C \rightarrow \mathbb{P}^1$ of degree 6).

Note that $Z = \{ L \in \Pic^{45}(C) : h^0(L) \geq 5, h^0(L-M) \geq 3 \text { and } h^0(L-2M) \geq 1 \}$ while $V^4_{45,3}(C,f) = \{ L \in \Pic ^{45}(C): h^0(L)=5, h^0(L-M)=3, h^0(L-2M)=1 \text { and } h^0(L-3M)=0 \}$.
Theorem \ref{TheoC} implies $Z$ is smooth along $V^4_{45,3} (C,f)$.
For a general element $L$ of $Z$ we also have $h^0(L+M)=7$.
Now Theorem \ref{theorem5.1} implies $L \in V^4_{45,3}(C,f)$ is a smooth point of $W^4_{45}(C)$ if and only if $h^0(L+M)=7$ and in case $h^0(L+M) >7$ then $L$ belongs to at least 2 irreducible components of $W^4_{45}(C)$.
As a matter of fact, in case $h^0(L+M)>7$ then $L$ also belongs to the component $\overline{\Sigma}_{\overrightarrow{g}}(C,f)$ with $\overrightarrow {g}=(-4,-4,-4,0,1,1)$ and having dimension 17.

Letting $\overrightarrow {e'}=(-4,-4,-4,-1,1,2)$ then if follows $\overline{\Sigma}_{\overrightarrow {e'}}(C,f) \cap V^4_{45,3}(C,f)$ (this is an open dense subset of $\overline{\Sigma}_{\overrightarrow {e'}}(C,f)$) is equal to $\Sing (W^4_{45}(C)) \cap V^4_{45}(C,f)$ and this has dimension 14.
On the other hand, letting $\overrightarrow {e''}=(-4,-4,-3,-2,1,2)$ then $\overset{\circ}{\Sigma}_{\overrightarrow {e''}[0]}(C,f)\subset V^4_{45,3}(C,f)$ and for $L \in \overset{\circ}{\Sigma}_{\overrightarrow {e''}[0]}(C,f)$ we have $h^0(L+M)=7$ hence $W^4_{45}(C)$ is smooth at $L$.
On the other hand, for a general element $L$ of $Z$ we have $h^0(L+2M)=9$ while $h^0(L+2M)>9$ for $L \in \overset{\circ}{\Sigma}_{\overrightarrow {e''}[0]}(C,f)$.
Notice that $\dim (\overset{\circ}{\Sigma}_{\overrightarrow{e''}[0]}(C,f))=18$.
\end{example}

\textbf {Acknowledgement.} I want to thank the authors of \cite{FFR} for answering some questions on the their paper.

\end{document}